\documentclass[12pt]{amsart}
\usepackage{graphicx}

\usepackage{amsmath, amsthm}
\usepackage{txfonts}

\newtheorem{theorem}{Theorem}[section]
\newtheorem{proposition}[theorem]{Proposition}
\newtheorem{lemma}[theorem]{Lemma}

\theoremstyle{definition}
\newtheorem{definition}[theorem]{Definition}

\theoremstyle{remark}
\newtheorem{remark}[theorem]{Remark}

\newcommand{\R}{\mathbb{R}}
\newcommand{\Z}{\mathbb{Z}}
\newcommand{\N}{\mathbb{N}}

\newcommand{\Int}{\mathrm{Int}\,}
\newcommand{\abs}[1]{\left|#1\right|}

\begin{document}
\title[A plat form presentation for surface-links]{A plat form presentation for surface-links}

\author{Jumpei Yasuda}
\address{Department of Mathematics, Graduate School of Science, Osaka University,  Toyonaka, Osaka 560-0043, Japan}

\email{u444951d@ecs.osaka-u.ac.jp}



\begin{abstract}
   In this paper, we introduce a method, called a plat form, of describing a surface-link in the 4-space using a braided surface.
   We prove that every surface-link, which is not necessarily orientable, can be described in a plat form.
   The plat index is defined as a surface-link invariant, which is an analogy of the bridge index for a link in the 3-space.
   We classify surface-links with plat index $1$ and show some examples of surface-links in plat forms.
\end{abstract}

\maketitle

\section{Introduction}\label{Section: Introduction}

In knot theory we often use two methods of presenting links in the $3$-space using braids: One is a closed braid form as in Figure~\ref{Figure: The closure of a braid}, and the other is a plat form as in Figure~\ref{Figure: A plat form of a link}.

\begin{figure}[h]
   \centering
   \begin{minipage}[t]{0.53\hsize}
      \centering
      \includegraphics[height = 20mm]{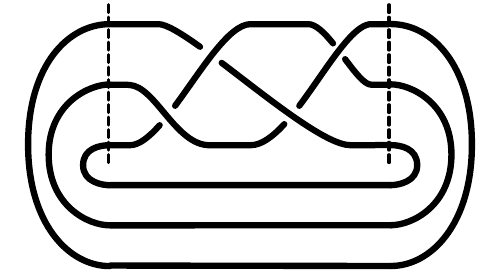}
      \caption{A closed braid form.}
      \label{Figure: The closure of a braid}
   \end{minipage}
   \begin{minipage}[t]{0.45\hsize}
      \centering
      \includegraphics[height = 20mm]{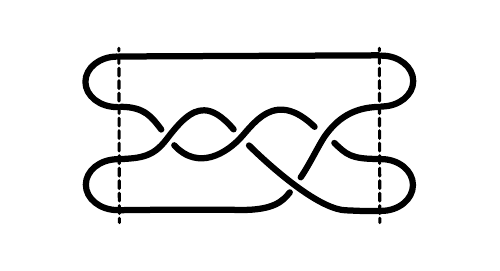}
      \caption{A plat form.}
      \label{Figure: A plat form of a link}
   \end{minipage}
\end{figure}

A \textit{surface-link} is a closed surface embedded in $\R^4$, and a $2$-knot is a $2$-sphere embedded in $\R^4$.
Two surface-links are considered to be equivalent if they are ambient isotopic in $\R^4$.
It is known that every orientable surface-link is equivalent to a surface-link in a closed $2$-dimensional braid form (cf. \cite{Kamada1992, Kamada2002_book, Viro90}).
It is an analogy of a closed braid form for a link.

The purpose of this paper is to introduce a new method of presenting a surface-link, which we call a plat form, as an analogy of a plat form for a link.

\begin{theorem}\label{Main theorem-A The existence of a plat form}
   Every surface-link is equivalent to a surface-link in a plat form.
\end{theorem}

We emphasize that our method works for every surface-link, while the closed $2$-dimensional braid form works only for orientable ones.
A \textit{genuine plat form} is a special case of a plat form.
Some surface-links can be presented in genuine plat forms.

\begin{theorem}\label{Main Theorem-B The existence of a genuine plat form}
   Every orientable surface-link is equivalent to a surface-link in a genuine plat form.
\end{theorem}

We show that the normal Euler number $e(F)$ of a surface-link $F$ in a genuine plat form is zero (Proposition~\ref{Proposition: normal Euler number of genuine plat form}).
It is unknown to the author whether every surface-link with $e(F) = 0$ is equivalent to one in a genuine plat form.

We define two surface-link invariants, which are called the \textit{plat index} and the \textit{genuine plat index}, denoted by $\mathrm{Plat}(F)$ and $\mathrm{g.Plat}(F)$, respectively.
These are analogies of the plat index, or the bridge index, of a link.

Using a theory of braided surfaces and $2$-dimensional braids, we show that a surface-link $F$ with $\mathrm{Plat}(F) = 1$ or with $\mathrm{g.Plat}(F) = 1$ is trivial (Theorem~\ref{Theorem The classification of trivial surface-knots}) and that a $2$-knot with $\mathrm{g.Plat}(F) = 2$ is ribbon (Theorem~\ref{Theorem genuine 2-plat 2-knot is ribbon}).
We also see an example of a $2$-knot whose plat index and genuine plat index are different (Proposition~\ref{Proposition plat index of 2-twist spun twist knot}).
An example of a non-trivial surface-link in a plat form is shown in Figure~\ref{Figure: mp of 2-twist spun trefoil} by using a motion picture (Proposition~\ref{Proposition plat index of 2-twist spun twist knot}).

\begin{figure}[h]
   \centering
   \includegraphics[width = 0.8\hsize]{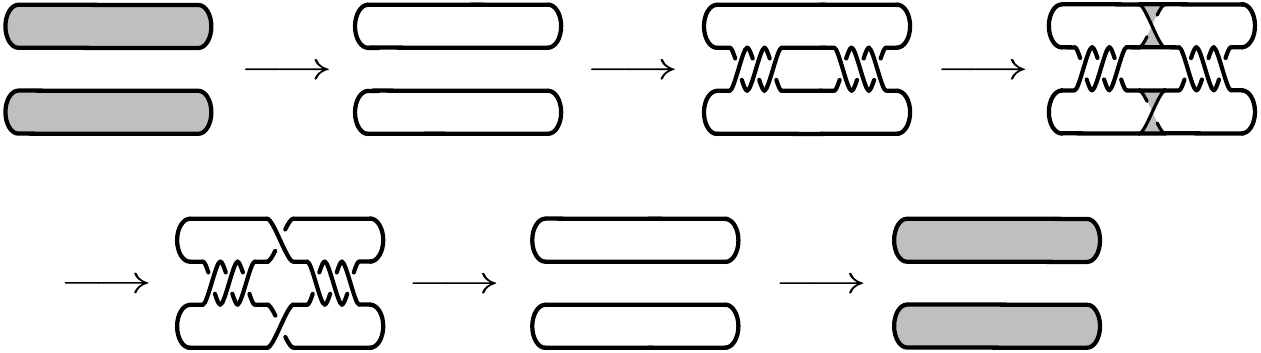}
   \caption{The $2$-twist spun trefoil in a (normal) plat form.}
   \label{Figure: mp of 2-twist spun trefoil}
\end{figure}

This paper is organized as follows.
In Section \ref{Section: Preliminaries}, we recall the notions of braids, surface-links, and braided surfaces.
We also recall the definition of a plat form for a link.
In Section \ref{Section: A plat form for a surface-link}, we define a (normal) plat form and a genuine plat form for a surface-link.
In Section \ref{Section: existence of plat forms}, we prove Theorems \ref{Main theorem-A The existence of a plat form} and \ref{Main Theorem-B The existence of a genuine plat form}.
In Section \ref{Section: The plat index of surface-links and examples}, we discuss the plat index and the genuine plat index of a surface-link, and show some examples.

We work in the PL or smooth category.
Surfaces embedded in the $4$-space are assumed to be locally flat in the PL category.

\section{Preliminaries}\label{Section: Preliminaries}
\subsection{A plat form presentation for a link}
Let $n$ be a positive integer, $I = [0,1]$ the interval, $D$ the square $I^2$ in $\R^2$, $\Int D$ the interior of $D$, and $Q_n =\{ q_1, \dots, q_n\}$ the subset of $n$ points in $D$ such that $q_k=\left({1}/{2},  {k}/{(n+1)} \right)$ for $k=1,2,\dots, n$.

An \textit{$n$-braid} is a union of $n$ intervals $\beta$ embedded in $D\times I$ such that each component intersects with every open disk $\Int D\times \{t\}$ ($t\in I$) transversely at a single point, and $\partial \beta = Q_n\times \{0,1\}$.
The $n$-braid group $B_n$ is the group consisting of equivalence classes of $n$-braids in $D\times I$.
The braid group $B_n$ is identified with the fundamental group $\pi_1(\mathcal{C}_n, Q_n)$ of the configuration space
$\mathcal{C}_n$ of $n$ points of $\Int D$.
We denote by $\sigma_1, \sigma_2, \dots, \sigma_{n-1}$ the standard generators of $B_n$ or their representatives due to Artin (\cite{Artin1925}).

To define the plat closure of a braided surface in Section~\ref{Section: A plat form for a surface-link}, we introduce the space of $m$ wickets.

\begin{definition}[\cite{Brendle-Hatcher2008}]
   A \textit{wicket} is a semicircle in $D \times I$ that meets $D \times \{0\}$ orthogonally at its endpoints in $\Int D\times \{0\}$.
   A \textit{configuration of $m$ wickets} is a disjoint union of $m$ wickets in $D\times I$.
   The \textit{space of $m$ wickets} $\mathcal{W}_m$ is the space consisting of all configurations of $m$ wickets.
\end{definition}

For a configuration $w= w_1 \cup \dots \cup w_m$ of $m$ wickets, we denote by $|\partial w|$ the $2m$ points $\partial w_1 \cup \dots \cup \partial w_m$ in $\Int D$, which is identified with $\Int D \times \{0\}$, and by $\partial w$ the $2m$ points $|\partial w|$ equipped with the partition $\{\partial w_1, \ldots, \partial w_m\}$.
Note that if two configurations $w$ and $w'$ satisfy $\partial w = \partial w'$, then $w=w'$.

The set $Q_{2m}$ equipped with the partition $\{ \{q_1, q_2\}, \dots, \{q_{2m-1}, q_{2m}\} \}$ bounds a unique configuration of $m$ wickets, which we call the \textit{standard configuration of $m$ wickets} and denote by $w_0$.

The fundamental group $\pi_1(\mathcal{W}_m, w_0)$ is called the \textit{wicket group} in \cite{Brendle-Hatcher2008}.
Let $|\partial|:  (\mathcal{W}_m, w_0) \to (\mathcal{C}_{2m}, Q_{2m})$ be the continuous map sending $w$ to $|\partial w|$.
It induces a homomorphism  $|\partial|_\ast:   \pi_1(\mathcal{W}_m, w_0) \to \pi_1(\mathcal{C}_{2m}, Q_{2m}) =B_{2m}$.

\textit{Hilden's subgroup} $K_{2m}$ is the subgroup of $B_{2m}$ generated by $\sigma_1$, $\sigma_2\sigma_1\sigma_3\sigma_2$, and $\sigma_{2i} \sigma_{2i-1} \sigma_{2i+1}^{-1} \sigma_{2i}^{-1}$ for $i = 1, \dots, m-1$ (\cite{Hilden1975}, cf. \cite{Birman1976}).

\begin{proposition}[\cite{Brendle-Hatcher2008}]\label{Proposition: Brendle-Hatcher2008}
   For each positive integer $m$, the homomorphism $|\partial|_\ast:   \pi_1(\mathcal{W}_m, w_0) \to \pi_1(\mathcal{C}_{2m}, Q_{2m}) =B_{2m}$ is injective and the image is Hilden's subgroup $K_{2m}$.
   Namely, the wicket group $\pi_1(\mathcal{W}_m, w)$ is isomorphic to Hilden's subgroup $K_{2m}$.
\end{proposition}

The isomorphism from $\pi_1(\mathcal{W}_m, w)$ to $K_{2m}$ is restated as follows:
Let $f: (I, \partial I)\to (\mathcal{W}_m, w_0)$ be a loop.
Consider a $2m$-braid $\beta_f = \bigcup_{t\in I} |\partial f(t)|\times \{t\} \subset D\times I$, then the isomorphism  sends $[f]\in \pi_1(\mathcal{W}_m, w)$ to $ [\beta_f] \in K_{2m}$.

\begin{definition}
   A loop $g: (I, \partial I) \to (\mathcal{C}_{2m}, Q_{2m})$ is \textit{liftable} if there exists a loop $f: (I, \partial I) \to (\mathcal{W}_m, w_0)$ such that $g = |\partial| \circ f$.
\end{definition}

\begin{definition}
   A $2m$-braid $\beta$ in $D\times I$ is \textit{adequate} or \textit{wicket-adequate} if the associated loop
   $g: (I, \partial I) \to (\mathcal{C}_{2m}, Q_{2m})$ is liftable, namely, there exists a loop $f: (I, \partial I)\to (\mathcal{W}_m, w_0)$ such that $\beta = \beta_f$.
\end{definition}

Note that Hilden's subgroup $K_{2m}$ consists of the elements of $B_{2m}$ represented by some adequate $2m$-braids.

Let $\beta$ be a $2m$-braid in $D \times I \subset \R^2 \times \R = \R^3$.
Attach a pair of the standard configurations of $m$ wickets to $\beta$ as in Figure~\ref{Figure: A plat form of trefoil}, and we obtain a link which is called the \textit{plat closure} of $\beta$ and denoted by $\widetilde{\beta}$.
A link is said to be \textit{in a plat form} when it is the plat closure of a braid.
Every link is equivalent to a link in a plat form.

\begin{figure}[h]
   \centering
   \includegraphics[width = 0.8\hsize]{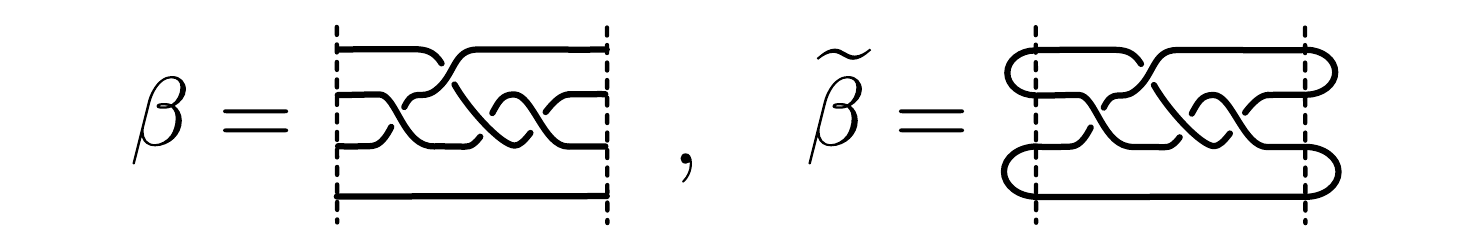}
   \caption{The plat closure of a braid.}
   \label{Figure: A plat form of trefoil}
\end{figure}

In Section~\ref{Section: A plat form for a surface-link} we introduce a plat form of a surface-link in $\R^4$.
We will also introduce a \textit{normal plat form}, which is a plat form satisfying a nice condition such that its motion picture is easy to describe.

To define a normal plat form of a surface-link in Section~\ref{Section: A plat form for a surface-link}, we construct an isotopic deformation changing the plat closure of an adequate braid to the plat closure of the trivial braid as follows:
Let $f: (I, \partial I)\to (\mathcal{W}_m, w_0)$ be a loop.
For each $t\in I$, let $\beta_t$ be $\bigcup_{s\in I} |\partial f((1-t)s)| \times \{s\}$ in $D\times I$, which is a union of $2m$ arcs.
We denote by $L_t$ a link obtained from $\beta_t$ by attaching the configuration $f(t)$ of $m$ wickets to the side of $D\times \{1\}$ and the standard configuration $w_0$ to the side of $D\times \{0\}$ in $\R^3$.
See Figure~\ref{Figure: mp of braid equipped with wickets 2}.
Then, $\{ L_t \}_{t\in I}$ is a $1$-parameter family of links in $\R^3$ such that $L_0$ is $\widetilde{\beta_f}$ and $L_1$ is the plat closure of the trivial $2m$-braid as in Figure~\ref{Figure: mp of braid equipped with wickets 2}.
We call $\{L_t\}_{t\in I}$ the \textit{isotopic deformation changing $\widetilde{\beta_f}$ to the plat closure of the trivial braid}.

As a corollary, the plat closure of an adequate $2m$-braid is an $m$-component trivial link.

\begin{figure}[h]
   \centering
   \includegraphics[width = 0.9\hsize]{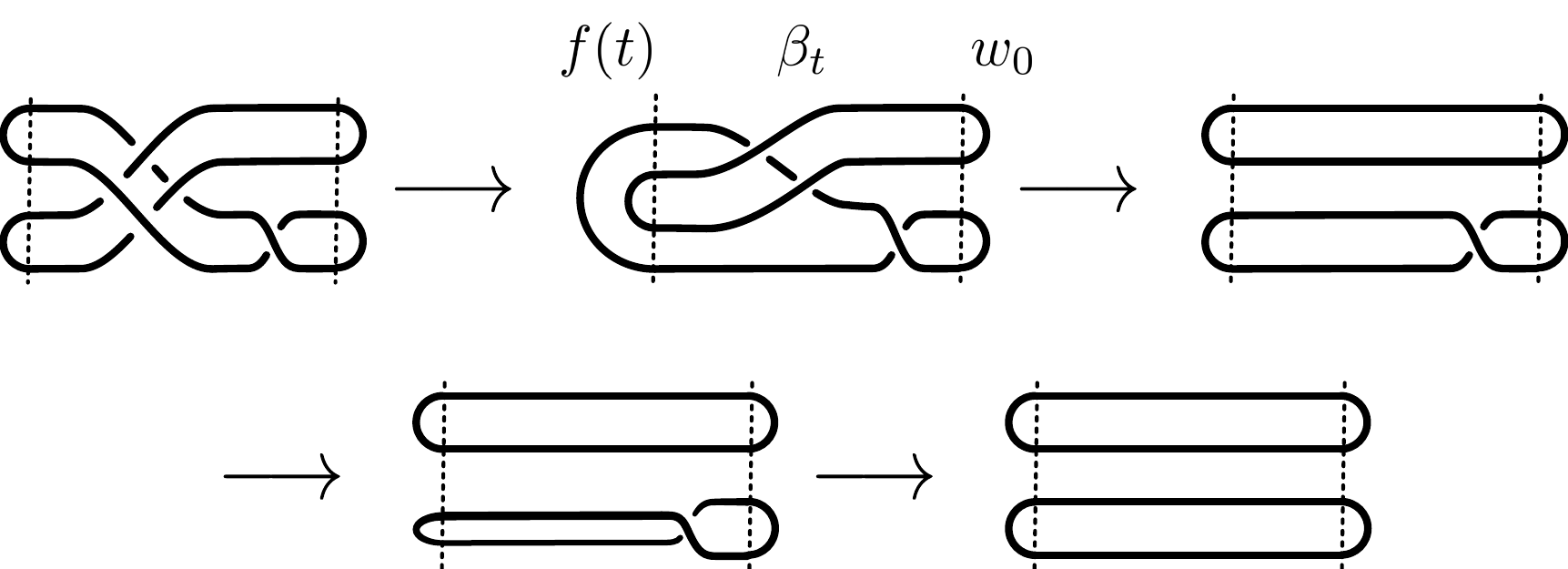}
   \caption{The isotopic deformation changing $\widetilde{\beta_f}$ to the plat closure of the trivial braid.}
   \label{Figure: mp of braid equipped with wickets 2}
\end{figure}

\subsection{Surface-links}
A \textit{surface-link} is a closed surface embedded in $\R^4$, and a \textit{surface-knot} is a connected surface-link.
A \textit{$2$-knot} is a surface-knot homeomorphic to a $2$-sphere.
A \textit{$2$-link} is a surface-link consisting of $2$-spheres.
Two surface-links $F$ and $F'$ are said to be \textit{equivalent} if they are ambient isotopic in $\R^4$.
We denote it by $F \simeq F'$ that $F$ and $F'$ are equivalent.

Let $h:\R^3\times \R^1 \to \R^1$ be the projection onto the second factor.
Set $F_{[t]} = F \cap \R^3\times \{t\}$ for $t\in \R$, which is called the \textit{cross-section} of $F$ at $t$.
A \textit{motion picture} of $F$ is a $1$-parameter family $\{F_{[t]}\}_{t\in \R}$.
We often describe surface-links using motion pictures.

A surface-knot is \textit{trivial} if it is equivalent to a connected sum of standardly embedded $2$-spheres, tori, and projective planes (\cite{Hosokawa-Kawauchi1979}).
Here standardly embedded projective planes $P_+$ and $P_-$ are illustrated in Figure~\ref{Figure: mp of P2}.

\begin{figure}[h]
   \centering
   \includegraphics[width = 0.9\hsize]{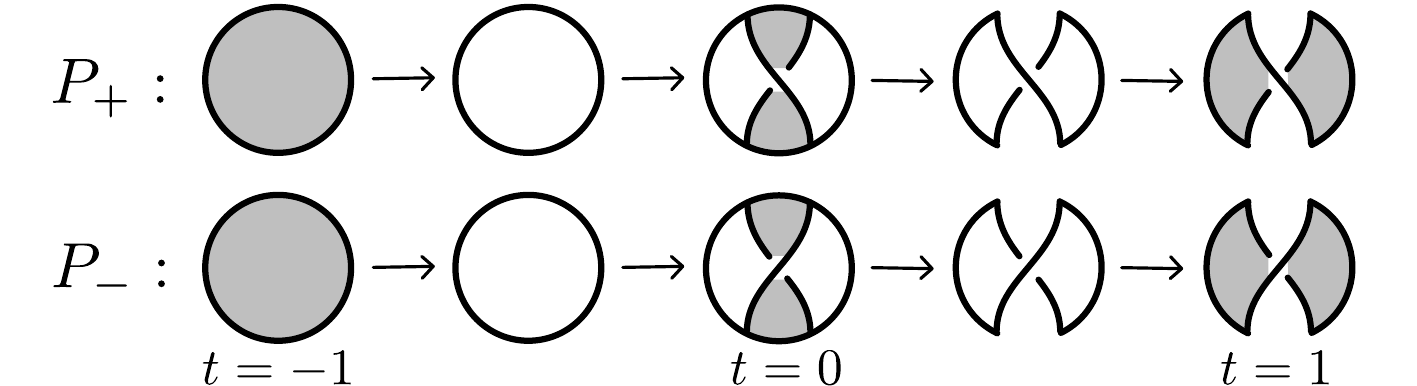}
   \caption{Motion pictures of $P_+$ and $P_-$.}
   \label{Figure: mp of P2}
\end{figure}

\subsection{Braided surfaces and $2$-dimensional braids}

A braided surface was introduced by Rudolph \cite{Rudolph1983} and a \textit{$2$-dimensional braid} was introduced by Viro (cf.~\cite{Kamada1994-08, Kamada1994-01, Kamada2002_book}).
Let $D_1$ and $D_2$ be the squares $I^2 \subset \R^2$ and $\mathrm{pr}_i : D_1\times D_2 \to D_i$ $(i=1,2)$ the projection onto the $i$-th factor.
Let $y_0 \in \partial D_2$ be a fixed base point.

\begin{definition}[\cite{Rudolph1983}, \cite{Viro90}]
   A (\textit{pointed}) \textit{braided surface} of degree $n$ is a surface $S$ embedded in $D_1\times D_2$ satisfying the following conditions:
   \begin{enumerate}
      \item $\pi_S = \mathrm{pr}_2|_S: S \to D_2$ is a simple branched covering map of degree $n$ (i.e., the preimage of each branch locus consists of $n-1$ points).
      \item $\partial S$ is the closure of an $n$-braid in the solid torus $D_1\times \partial D_2$.
      \item $\mathrm{pr_1}(\pi_S^{-1}(y_0))=Q_n$.
   \end{enumerate}
   In particular, a \textit{$2$-dimensional braid} of degree $n$ is a braided surface $S$ of degree $n$ such that $\partial S$ is trivial, i.e., $\mathrm{pr_1}(\pi_S^{-1}(y))=Q_n$ for all $y \in \partial D_2$.
\end{definition}

The degree of $S$ is denoted by $\deg S$.
We say that two braided surfaces of the same degree are \textit{equivalent} if they are ambient isotopic by an isotopy $\{h_s\}_{s\in I}$ of $D_1\times D_2$ such that each $h_s$ ($s\in I$) is fiber-preserving when we regard $D_1 \times D_2$ as the trivial $D_1$-bundle over $D_2$, and the restriction of $h_s$ to ${\mathrm{pr}_2^{-1}(y_0)}$ is the identity map.
A braided surface is \textit{trivial} if it is equivalent to $Q_n\times D_2$.

\begin{lemma}[cf. \cite{Kamada2002_book}]\label{Lemma: Triviality of braided surface}
   A braided surface $S$ is trivial if and only if $S$ has no branch points.
\end{lemma}

We assume $D_1\times D_2 \subset \R^2 \times \R^2 = \R^4$.
Let $S$ be a $2$-dimensional braid of degree $n$.
The \textit{closure} of $S$ is an orientable surface-link in $\R^4$ obtained from $S$ by attaching $n$ $2$-disks trivially outside $D_1 \times D_2$ in $\R^4$ along the boundary $\partial S$.
It is described in Figure~\ref{Figure: mp of the closure of 2-dim braid} when $n=3$, where $\varepsilon$ is a positive number and $S_{[t]}=S\cap D_1\times (I\times \{t\})$ $(t \in I)$.

\begin{proposition}[\cite{Kamada1994-01, Viro90}]\label{Proposition: Alexander theorem wrt closed 2-dim braids}
   Every orientable surface-link is equivalent to the closure of a $2$-dimensional braid.
\end{proposition}

\begin{figure}[h]
   \centering
   \includegraphics[width=\hsize]{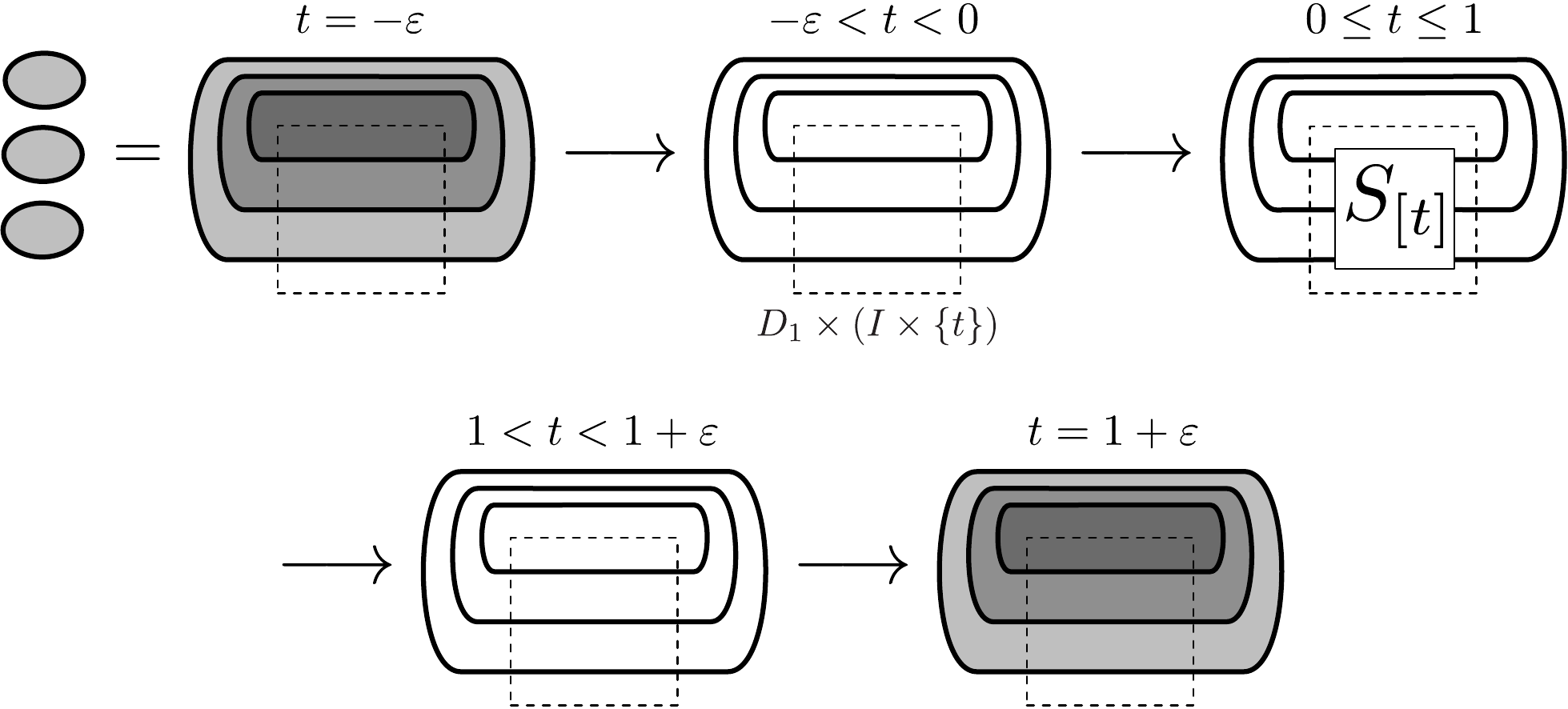}
   \caption{The closure $\overline{S}$ of a $2$-dimensional braid $S$.}
   \label{Figure: mp of the closure of 2-dim braid}
\end{figure}

For an orientable surface-link $F$, the \textit{braid index} of $F$,
denoted by $\mathrm{Braid}(F)$, is the minimum degree of $2$-dimensional braids whose closures are equivalent to $F$.

\section{A plat form presentation for a surface-link}\label{Section: A plat form for a surface-link}

In this section, we introduce a plat form for a surface-link.

We fix a loop $\mu: (I, \partial I) \to (\partial D_2, y_0)$ which runs once on $\partial D_2$ counter-clockwise.
For a braided surface $S$ of degree $n$, let
$g_S : (I, \partial I) \to (\mathcal{C}_{n}, Q_{n})$ be a loop in the configuration space $\mathcal{C}_n$ obtained by
\[
   g_S (t) ~=~ \mathrm{pr}_1(\pi_S^{-1}(\mu(t)))
\]
and $\beta_S$ an $n$-braid in $D_1\times I$ obtained by
\[
   \beta_S ~=~ \bigcup_{t\in I} \mathrm{pr}_1(\pi_S^{-1}(\mu(t)))\times \{t\},
\]
where $\pi_S: S \to D_2$ is the simple branched covering map appearing in the definition of a braided surface.
Then $\partial S$ is the closure of $\beta_S$ in $D_1\times \partial D_2$.

\begin{definition}\label{Def:adequate}
   A braided surface $S$ in $D_1 \times D_2$ is \textit{adequate} if $g_S$ is liftable or equivalently if $\beta_S$ is adequate.
\end{definition}

Note that the degree of an adequate braided surface is even.
For an adequate braided surface $S$ of degree $2m$, let $f_S: (I, \partial I) \to (\mathcal{W}_m, w_0)$ be the lift of $g_S$, i.e., a loop in $\mathcal{W}_m$ with $g_S = |\partial| \circ f_S$.

Let $N$ be a regular neighborhood of $\partial D_2$ in $\R^2\setminus \Int D_2$.
Since $N$ is homeomorphic to an annulus $I\times S^1$, we identify them by a fixed identification map $\phi: I \times S^1 \to N$ such that $\phi(0,p(t)) = \mu(t) \in \partial D_2$ for all $t \in I$, where $p: I \to S^1 = I/\partial I$ is the quotient map.

\begin{definition}
A properly embedded surface $A$ in $D_1\times N$ is \textit{of wicket type} if there exists
a loop $f: (I, \partial I) \to (\mathcal{W}_m, w_0)$ such that
 \[
          A ~=~ \bigcup_{t\in I} f(t)\times \{p(t)\} ~\subset~ (D_1 \times I) \times S^1 ~=~ D_1\times N.
   \]
In this case, we say that $A$ is \textit{associated with} $f$ and denote it by $A_f$.
\end{definition}

We remark that a surface $A$ of wicket type is a union of annuli or M\"{o}bius bands, and that $\partial A= \partial A_f$ is expressed as
\[
   \partial A ~=~ \bigcup_{t\in I} |\partial f(t)|\times \{p(t)\} ~\subset~ D_1 \times S^1 ~=~ D_1 \times \partial D^2.
\]
Since two loops $f$ and $f'$ in $(\mathcal{W}_m, w_0)$ with $|\partial| \circ f = |\partial| \circ f'$ are the same, we see that two surfaces $A$ and $A'$ of wicket type with $\partial A = \partial A'$ are the same.

Let $S$ be an adequate braided surface, and let $f: (I, \partial I) \to (\mathcal{W}_m, w_0)$ be a loop with $g_S = |\partial| \circ f$.
Then it holds that $S \cap A_f = \partial S = \partial A_f$.
We denote $A_f$ by $A_S$ and say that $A_S$ is the \textit{surface of wicket type associated with $S$}.

\begin{definition}
   Let $S$ be an adequate braided surface and $A_S$ the surface of wicket type associated with $S$.
   The \textit{plat closure of $S$}, denoted by $\widetilde{S}$, is the union of $S$ and $A_S$ in $\R^4$.
\end{definition}

When $\deg S = 2m$ and $S$ has $r$ branch points, the Euler characteristic $\chi(S)$ of $S$ is $2m-r$.
Since $\chi(A_S) = \chi(\partial A_S) = 0$, we have $\chi(\widetilde{S}) = 2m-r$.

\begin{definition}
A surface-link is said to be \textit{in a plat form} if it is the plat closure of an adequate braided surface.
Moreover, a surface-link is said to be \textit{in a genuine plat form} if it is that of a $2$-dimensional braid.
\end{definition}

We introduce a \textit{normal plat form} for a surface-link by using a motion picture as follows:
Let $\widetilde{S}$ be the plat closure of an adequate braided surface $S$ of degree $2m$, and set $\widetilde{S}_{[t]} = \widetilde{S} \cap \R^3\times \{t\}$ ($t\in \R$) and $S_{[t]} = S \cap D_1\times (I\times \{t\}) = S \cap \R^3\times \{t\}$ ($t \in [0,1]$).
Replacing $S$ with an equivalent braided surface if necessary, we may assume that $S$ satisfies the following conditions for some $t_0 \in [0,1]$:
\begin{enumerate}
   \item $S$ has no branch points over $I\times [t_0,1] \subset D_2$.
   \item $\mathrm{pr}_1(\pi_S^{-1}(y)) = Q_{2m}$ for every $y \in \partial D_2 \setminus \left(\{1\} \times [t_0, 1]\right)$.
   \item $S_{[t_0]} = \beta_S \times \{t_0\}$.
\end{enumerate}
In particular, $S_{[0]}$ and $S_{[1]}$ are both the trivial braids.
Furthermore, replacing $S$ with an equivalent braided surface if necessary, we may assume that the motion picture $\{\widetilde{S}_{[t]}\}_{t\in [t_0, 1]}$ between $t = t_0$ and $t = 1$ is the isotopic deformation changing $\widetilde{\beta_f}$ to the plat closure of the trivial braid.
(See Figure~\ref{Figure: mp of braid equipped with wickets 2}.)
Finally, deforming $A_S$ by an ambient isotopy rel boundary, we have a surface-link $F$, equivalent to $\widetilde{S}$, described by a motion picture as in Figure~\ref{Figure: mp of plat cl for braided surface}.
The surface-link $F$ in this form is said to be in a \textit{normal plat form}.

\begin{figure}[h]
   \centering
   \includegraphics[width = 0.9\hsize]{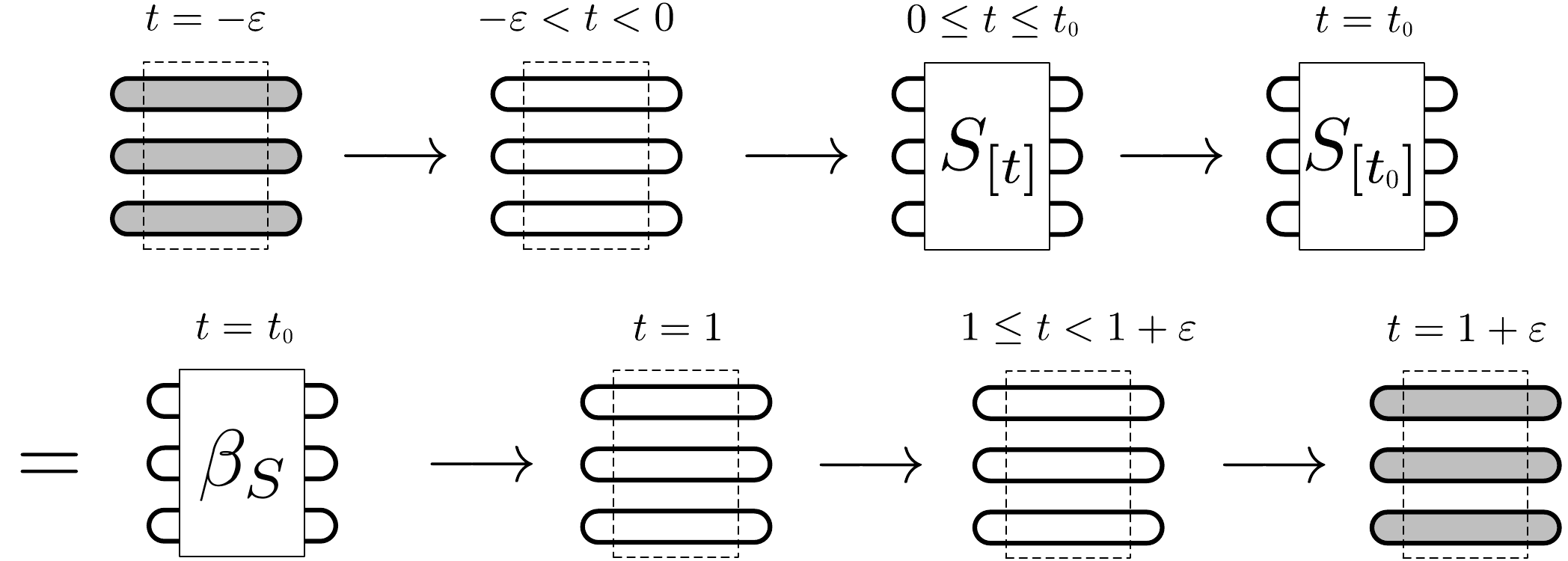}
   \caption{A surface-link in a normal plat form.}
   \label{Figure: mp of plat cl for braided surface}
\end{figure}


\section{Proofs of Theorems \ref{Main theorem-A The existence of a plat form} and \ref{Main Theorem-B The existence of a genuine plat form}}\label{Section: existence of plat forms}
In this section, we give proofs of Theorems \ref{Main theorem-A The existence of a plat form} and \ref{Main Theorem-B The existence of a genuine plat form}.
To prove them, we discuss a plat form for a link and a banded link presentation for a surface-link.

\subsection{Stabilization and generalized stabilization for braids}

For positive integers $n$ and $n'$ with $n \leq  n'$,
let $\iota_n^{n'}: B_n \to B_{n'}$ denote the natural inclusion map from $B_n$ to $B_{n'}$ sending each generator $\sigma_i\in B_n$ to $\sigma_i\in B_{n'}$.

A \textit{stabilization} of a $2m$-braid $\beta$ is a replacement of $\beta$ with a $2m'$-braid $\beta'$
such that
\[
   \beta' ~=~ \iota_{2m}^{2m'}(\beta) \, \sigma_{2m}\,\sigma_{2(m+1)}\,\sigma_{2(m+2)}  \ldots\,\sigma_{2(m'-1)},
\]
where $m'$ is an integer with $m \leq m'$.
We also call a stabilization an \textit{$l$-stabilization} when $l = m'-m$.

It is obvious that if $\beta'$ is obtained from $\beta$ by an $l$-stabilization then the plat closure of $\beta'$ is equivalent to that of $\beta$ as links in $\R^3$.
See Figure~\ref{Figure: Stabilization of a braid} for $l=1, 2$.

\begin{figure}[h]
   \centering
   \includegraphics[height=25mm]{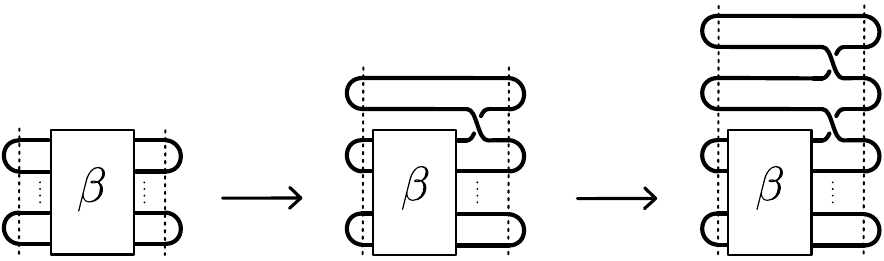}
   \caption{Plat closures of stabilized braids.}
   \label{Figure: Stabilization of a braid}
\end{figure}

\begin{proposition}[\cite{Birman1976}]\label{Proposition: Birman's criterion for knots in plat forms}
   Let $\beta_i$ ($i=1,2$) be a $2m_i$-braid such that the plat closure $\widetilde{\beta_i}$ is a knot.
   Then $\widetilde{\beta_1}$ is equivalent as knots in $\R^3$ to $\widetilde{\beta_2}$ if and only if there exists an integer $t \geq \max\{m_1,m_2\}$ such that for each $m \geq t$, the $2m$-braids $\beta_i'$ ($i=1,2$) obtained from $\beta_i$ by stabilization belong to the same double coset of $B_{2m}$ modulo $K_{2m}$.
\end{proposition}

Proposition~\ref{Proposition: Birman's criterion for knots in plat forms} is generalized into the case of links in $\R^3$ by using the notion of a generalized stabilization.

Let $\Lambda_m$ be the set of $m$-tuples of non-negative integers.
For two elements $\lambda=(l_1, \ldots, l_m)$ and $\lambda'=(l'_1, \ldots, l'_m)$ of $\Lambda_m$, we write $\lambda \preceq \lambda'$ if $l_i \leq l_i'$ for each $i = 1, \ldots, m$.
Then $\preceq$ is a (directed) partial ordering on $\Lambda_m$.
Put $|\lambda|_0 = m$, $|\lambda|_i =  m + l_1 + \cdots + l_i$ ($i=1,\ldots,m$), and $|\lambda| = |\lambda|_m$.
For a given $\lambda \in \Lambda_m$, we denote $\tau_i = \sigma_{2i} \sigma_{2i-1} \sigma_{2i+1}\sigma_{2i} \in K_{2|\lambda|}$ ($1\leq i \leq |\lambda|-1$) and
\[
   T_{i,\, j} = \prod_{k=i}^{m-1} \tau_{k} \cdot \prod_{k=m}^{j} \tau_{k}^{-1} \in K_{2|\lambda|} \quad (1 \leq i\leq m, m-1 \leq j \leq |\lambda|),
\]
where the former or later product is assumed to be the identity element of the group if $i=m$ or $j=m-1$, respectively, 
and we construct a $2\abs{\lambda}$-braid $T(\lambda)$ as follows:
\[
   T(\lambda) ~=~ \prod_{i=1}^{m} T_{i,\, (|\lambda|_{(i-1)}-1)}\,
   \sigma_{2|\lambda|_{i-1}}\sigma_{2(|\lambda|_{i-1}+1)}\cdots \sigma_{2|\lambda|_{i}}\,
   T_{i,\, (|\lambda|_{(i-1)}-1)}^{-1}.
\]


For a $2m$-braid $\beta$ and $\lambda \in \Lambda_{m}$, we let $\beta^\lambda$ denote a $2\abs{\lambda}$-braid such that
$$\beta^\lambda = \iota_{2m}^{2\abs{\lambda}}(\beta)\cdot T(\lambda).$$

A \textit{generalized stabilization} (with respect to $\lambda$) or \textit{$\lambda$-stabilization} of $\beta$ is a replacement of $\beta$ with $\beta^\lambda$.
A $\lambda$-stabilization is a composition of $l_i$-stabilization performed on the $2i$-th strand of $\beta$ for each $i =1, \ldots, m$.
A $l$-stabilization of $\beta$ is a $\lambda$-stabilization with $\lambda = (0, \ldots, 0,l) \in \Lambda_{m}$.
Figure~\ref{Figure: an example of generalized stabilization} depicts the plat closure of a $12$-braid obtained from a $6$-braid $\beta$ by a generalized stabilization with respect to $\lambda = (2, 0, 1) \in \Lambda_3$.

\begin{figure}[h]
   \centering
   \includegraphics[width = 0.7\hsize]{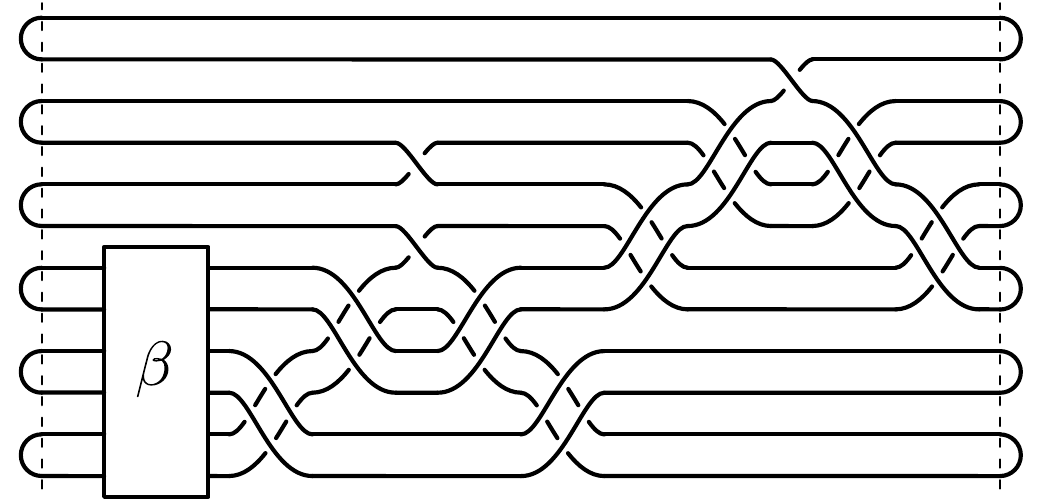}
   \caption{The plat closure of a $(2,0,1)$-stabilized braid.}
   \label{Figure: an example of generalized stabilization}
\end{figure}

The following proposition states that two braids of even degrees have equivalent plat closures as links in $\R^3$ if and only if, after applying a generalized stabilization suitably,
they belong to the same double coset of $B_{2m}$ modulo $K_{2m}$.

\begin{proposition}[cf. \cite{Birman1976}]\label{Proposition: Birman's criterion for links in plat forms}
   Let $\beta_i$ ($i=1,2$) be a $2m_i$-braid.
   The plat closure $\widetilde{\beta_1}$ is equivalent to $\widetilde{\beta_2}$ as links in $\R^3$ if and only if there exists an element $\lambda \in \Lambda_{m_1}$ satisfying the following condition:
   For any $\lambda_1 \succeq \lambda$, there exists $\lambda_2 \in \Lambda_{m_2}$ with $|\lambda_1| = |\lambda_2|$ such that $\beta_1^{\lambda_1}$ and $\beta_2^{\lambda_2}$
    belong to the same double coset of $B_{2m}$ modulo $K_{2m}$, where $m = |\lambda_1| = |\lambda_2|$.
\end{proposition}

Proposition~\ref{Proposition: Birman's criterion for links in plat forms} is proved directly by applying the proof of Proposition~\ref{Proposition: Birman's criterion for knots in plat forms} given in \cite{Birman1976} for each component of a link.

\subsection{A banded link presentation for a surface-link}
A \textit{banded link} in $\R^3$ means a pair $(L, B)$ of a link $L$ and a family $B$ of mutually disjoint bands attaching to $L$.
We let $L_B$ denote the link obtained from $L$ by surgery along the bands belonging to $B$.
A banded link $(L, B)$ is \textit{admissible} if both $L$ and $L_B$ are trivial links.

Let $(L, B)$ be an admissible banded link in $\R^3$. Let $\mathbf{d}$ and $\mathbf{D}$ be unions of mutually disjoint 2-disks embedded in $\R^3$ bounded by $L$ and $L_B$, respectively.
Consider a closed surface $F = F(L, B)$ in $\R^4 = \R^3 \times \R$ defined by
\begin{align*}\label{Equation motion picture}
     p(F\cap \R^3\times\{t\}) =
   \begin{cases}
    \mathbf{D}  & (t=1), \\
    L_B & (0<t<1), \\
    L \cup |B|   & (t=0), \\
    L  & (-1<t<0), \\
    \mathbf{d} & (t=-1),\mathrm{and}\\
    \emptyset  & \textrm{otherwise},
   \end{cases}
\end{align*}
where $|B|$ is the union of the bands belonging to $B$.
We call $F(L, B)$ a \textit{closed realizing surface of $(L, B)$}.
Although it depends on a choice of $\mathbf{d}$ and $\mathbf{D}$, the equivalence class as surface-links does not depend on them (cf. \cite{Kamada2017_book, K-S-S1982}).

Let $r$ be a real number, and let $h: \R^3\times(-\infty,r] \to (-\infty,r]$ be the projection onto the second factor, which we regard as a height function of $\R^3\times(-\infty,r]$.

\begin{lemma}[cf. \cite{Kamada2017_book, K-S-S1982}]\label{Lemma: 2017_book-thm3.3.2}
   Let $F$ and $F'$ be compact surfaces properly embedded in $\R^3\times (-\infty,r]$ such that all critical points of $F$ and $F'$ are minimal points with respect to $h$, and their boundaries are the same trivial link in $\R^3\times \{r\}$.
   Then, $F$ and $F'$ are ambient isotopic in $\R^3\times (-\infty,r]$ rel $\R^3\times \{r\}$.
\end{lemma}



\begin{lemma}[\cite{K-S-S1982}]\label{Lem:realizing_surface_1}
   If two admissible banded links $(L, B)$ and $(L', B')$ are ambient isotopic in $\R^3$, then their closed realizing surfaces $F(L, B)$ and $F(L', B')$ are equivalent.
\end{lemma}

\begin{lemma}[\cite{K-S-S1982}]\label{Lem:realizing_surface_2}
   Any surface-link $F$ is equivalent to a closed realizing surface $F(L, B)$ of an admissible banded link $(L, B)$.
\end{lemma}

\begin{lemma}\label{Lemma: Deformation to normal banded braid form}
By an isotopy of $\R^3$, any banded link $(L,B)$ in $\R^3$ is deformed to a banded link $(L_0, B_0)$ satisfying the following conditions:
   \begin{enumerate}
      \item There exists a disk $D$ in $\R^2$ and a $2m_0$-braid $\beta_0$ in $D\times I$ ($\subset \R^2 \times \R = \R^3$) for some $m_0\in \N$ such that $\beta_0 = L_0\cap D\times I$ and $\widetilde{\beta_0} = L_0$.
      \item There exist mutually disjoint $n$ subcylinders $U_i = d_i \times [s_i, t_i]$ ($i = 1, \ldots, n$) in $D \times I$ such that each $U_i$ contains a part of $L_0$ as a pair of vertical line segments and a half-twisted band $b_i \in B_0$ as in Figure~\ref{Figure: Thm1-7 band in cylinder}, where $n$ is the number of bands belonging to $B_0$.
   \end{enumerate}

   Furthermore, we may take subcylinders $U_i = d_i \times [s_i, t_i]$ such that $d_1, \dots, d_n$ are mutually disjoint disks in $\Int D$ and $[s_i, t_i] = [2/5, 3/5]$ for $i = 1, \dots, n$.
\end{lemma}

\begin{figure}[h]
   \centering
   \includegraphics[height = 27mm]{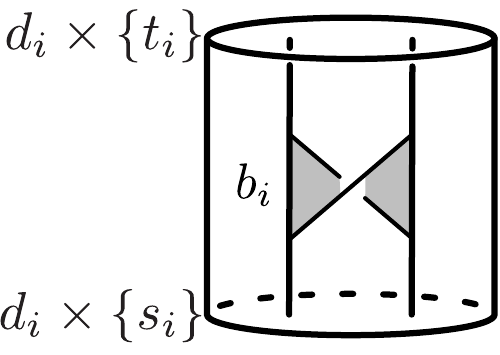}
   \caption{A local model of $L_0$ and $b_i$ in $U_i = d_i \times [s_i, t_i]$.}
   \label{Figure: Thm1-7 band in cylinder}
\end{figure}

\begin{proof}
   Let $d_1, \dots, d_n$ be mutually disjoint disks in $\Int D$ and let $U_i = d_i \times [2/5, 3/5]$ for $i = 1, \dots, n$.
   By an isotopy of $\R^3$, $(L, B)$ is deformed into $(L_1, B_0)$ such that for each $i$, $U_i$ intersects with $(L_1, B_0)$ as in Figure~\ref{Figure: Thm1-7 band in cylinder}.

   By an isotopy of $\R^3$ keeping $U_i$ ($i = 1, \dots, n$) fixed pointwise, $(L_1, B_0)$ is deformed into $(L_2, B_0)$ such that all maximal points of $L_2$ are in $\R^2 \times \{1\}$ and all minimal points of $L_2$ are in $\R^2 \times \{0\}$.
   Finally, by an isotopy of $\R^3$ keeping $U_i$ ($i = 1, \dots, n$) fixed pointwise, we deform the link $L_2$ into a link $L_0$ satisfying the condition (1).
\end{proof}

We denote by $(\beta_0)_{B_0}$ the $2m_0$-braid in $D\times I$ obtained from $\beta_0$ by surgery along bands belonging to $B_0$.

\begin{proof}[Proof of Theorem~\ref{Main theorem-A The existence of a plat form}]
We prove the theorem by $3$ steps.
Let $F$ be a surface-link.

\textbf{Step $1$}:
By Lemmas~\ref{Lem:realizing_surface_1},~\ref{Lem:realizing_surface_2} and~\ref{Lemma: Deformation to normal banded braid form}, $F$ is equivalent to a closed realizing surface of a banded link $(L_0, B_0)$ satisfying the conditions (1) and (2) in Lemma~\ref{Lemma: Deformation to normal banded braid form}.
Let $\beta_0$ be the $2 m_0$-braid in $D\times I$ as in Lemma~\ref{Lemma: Deformation to normal banded braid form}.

Let $c_1$ and $c_2$ be the numbers of components of $L_0$ and $(L_0)_{B_0}$, respectively.
Since $\widetilde{\beta_0} = L_0$ is a trivial link of $c_1$ components, the plat closure $\widetilde{\beta_0}$ is equivalent as links in $\R^3$ to the plat closure $\widetilde{1_{2c_1}}$ of the trivial braid $1_{2c_1} \in B_{2 c_1}$.
The plat closure $\widetilde{(\beta_0)_{B_0}}$ is equivalent to the plat closure $\widetilde{1_{2c_2}}$ of the trivial braid $1_{2c_2} \in B_{2 c_2}$ by the same reason.

Applying Proposition~\ref{Proposition: Birman's criterion for links in plat forms} to the two pairs $(\beta_0, 1_{2c_1})$ and $((\beta_0)_{B_0}, 1_{2c_2})$ of braids in $D\times I$, there exist a positive integer $m \in \Z$, three elements $\lambda\in \Lambda_{m_0}$, $\lambda_1\in \Lambda_{c_1}$, $\lambda_2\in \Lambda_{c_2}$, and four adequate $2m$-braids $\gamma$, $\gamma'$, $\delta$, $\delta'$ in $D\times I$ such that $|\lambda| = |\lambda_1| = |\lambda_2| = m$ and
\[
   \beta_1 ~=~ \gamma\,\alpha_1\,\gamma', \quad
   \beta_2 = \delta\,\alpha_2\,\delta' \quad \mathrm{in} \quad B_{2m},
\]
where $\beta_1 = \beta_0^\lambda$, $\alpha_1 = 1_{2c_1}^{\lambda_1}$, $\beta_2 = (\beta_0)_{B_0}^{\, \lambda}$ and $\alpha_2 = 1_{2c_2}^{\lambda_2}$ are $2m$-braids in $D\times I$ obtained
by generalized stabilization.


Since $\beta_1$ is a $\lambda$-stabilized $\beta_0$, there exists a subcylinder $U$ of $D\times I$ such that $\beta_1 \cap U = \beta_0$ under an identification of $U$ and $D\times I$.
Let $B_1$ be the set of bands attaching to $\beta_1$ obtained from $B_0$ via the identification.
Then, $\beta_2$ and $(\beta_1)_{B_1}$ are the same braid.
Note that $(\widetilde{\beta_1}, B_1)$ is ambient isotopic to $(L_0, B_0)$.

\textbf{Step $2$}: We construct a properly embedded compact surface $S_0$ in $D_1 \times D_2$ and a braided surface $S$ of degree $2m$ in $D_1\times D_2$.
Let $0=t_0 < t_1 < \dots < t_6 < t_7 = 1$ be a partition of $I=[0,1]$.
We divide $D_2=I\times I$ into seven pieces $E_0, \dots, E_6$ with $E_i = I\times [t_i,t_{i+1}]$.
Let $\alpha_1^*$ and $\alpha_2^*$ be $2m$-braids in $D_1 \times I$ given by
\[
   \alpha_1^* ~=~ \prod_{i=1}^{m} T_{i,\, (|\lambda_1|_{(i-1)}-1)} T_{i,\, (|\lambda_1|_{(i-1)}-1)}^{-1}, \quad
   \alpha_2^* ~=~ \prod_{i=1}^{m} T_{i,\, (|\lambda_2|_{(i-1)}-1)} T_{i,\, (|\lambda_2|_{(i-1)}-1)}^{-1},
\]
which are obtained from $\alpha_1 = 1_{2c_1}^{\lambda_1} = T(\lambda_1)$ and $\alpha_2 = 1_{2c_2}^{\lambda_2} = T(\lambda_2)$ by removing the parts $\sigma_{2|\lambda_1|_{(i-1)}}\sigma_{2(|\lambda_1|_{(i-1)} + 1)} \dots \sigma_{2|\lambda_1|_{i}}$ and $\sigma_{2|\lambda_2|_{(i-1)}}\sigma_{2(|\lambda_2|_{(i-1)} + 1)} \dots \sigma_{2|\lambda_2|_{i}}$ ($i = 1, \dots, m$), respectively (Figure~\ref{Figure: Thm1-9 motion picture oover E1}).
Note that $\alpha_1^*$ and $\alpha_2^*$ are equivalent to the trivial braid $1_{2m} = Q_{2m} \times I$ as braids in $D_1 \times I$.

\begin{figure}[h]
    \centering
    \includegraphics[width = 0.9\hsize]{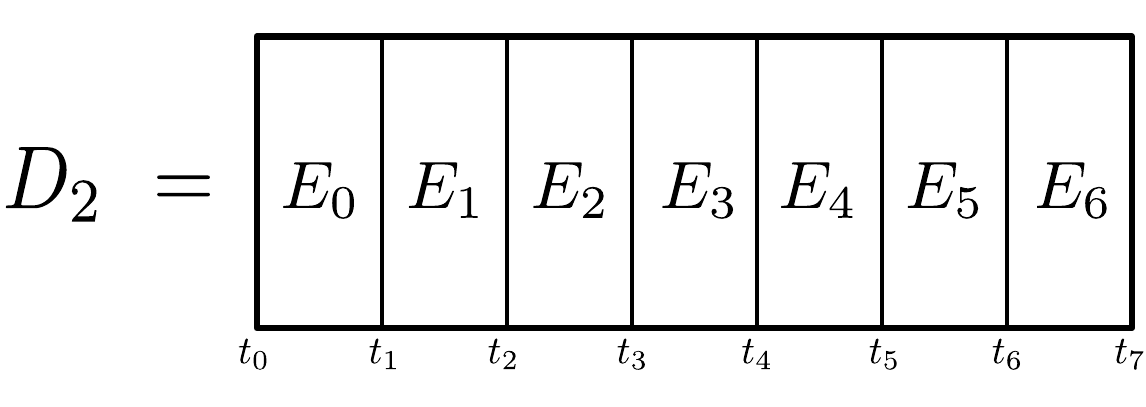}
    \caption{The partition of $D_2$.}
    \label{Figure: Thm1-5 Partition of D2}
\end{figure}

Let $\mathrm{p}_1: D_1\times I \to D_1$ and $\mathrm{p}_2: D_1\times I \to I$ be the projections onto the first and second factors, respectively.
Let $\mathrm{pr}_i: D_1\times D_2 \to D_i$ be the projections onto the $i$-th factors ($i = 1, 2$).
For a braid $b$ in $D_1 \times I$ and $s \in I$, we denote by $b_{[s]}$ the image $\mathrm{p}_1(b\cap \mathrm{p}_2^{-1}(s))$ in $D_1$ of the intersection $b \cap p_2^{-1}(s)$.

Now, we define a properly embedded compact surface $S_0$ in $D_1\times D_2$, step by step, as follows:
\begin{enumerate}
   \setcounter{enumi}{-1}
   \setlength{\leftskip}{-6mm}
   \item First, we define $S_0 \cap D_1\times \partial E_0$ by
   \begin{align*}
      \mathrm{pr_1}(S_0\cap\mathrm{pr}_2^{-1}(s,t)) ~=~
      \begin{cases}
         (\alpha_1^*)_{[s]}   \quad &((s,t)\in I\times \{t_1\}),\\
         Q_{2m}               \quad &((s,t)\in \{0,1\}\times [t_0,t_1]),\\
         Q_{2m}               \quad &((s,t)\in I\times \{t_0\}).
      \end{cases}
   \end{align*}
   See Figure~\ref{Figure: Thm1-6 Putting up braids on D2}.
   Since $\alpha_1^*$ is equivalent to the trivial $2m$-braid, we may define $S_0 \cap D_1\times E_0$ as a braided surface of degree $2m$ without branch points in $D_1 \times E_0$, which is trivial by Lemma~\ref{Lemma: Triviality of braided surface}.

   \begin{figure}[h]
      \centering
      \includegraphics[width = 0.9\hsize]{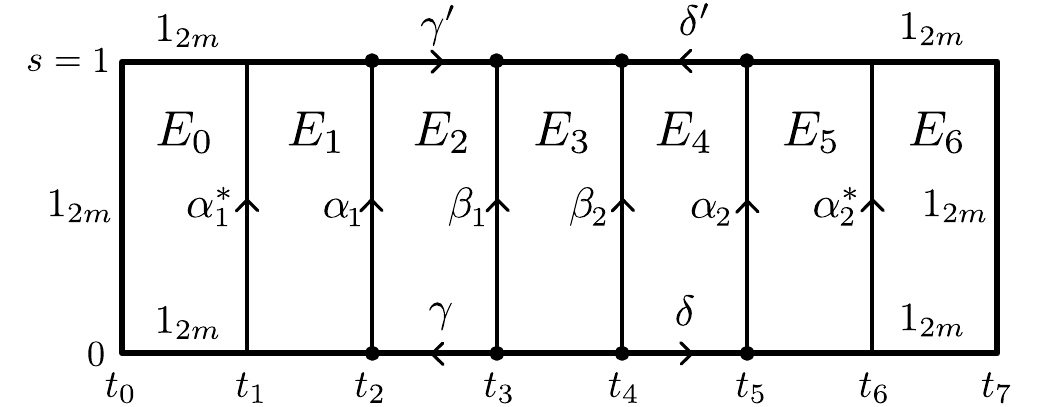}
      \caption{A blueprint for a surface $S_0$. Each braid is appeared as the section of $S_0$.}
      \label{Figure: Thm1-6 Putting up braids on D2}
   \end{figure}

   \item We define $S_0 \cap D_1\times (E_1 \setminus I\times \{(t_1+t_2)/2\})$ as follows:
   \begin{align*}
      \mathrm{pr_1}(S_0 \cap\mathrm{pr}_2^{-1}(s,t)) ~=~
      \begin{cases}
         (\alpha_1^*)_{[s]}  \quad &((s,t)\in I\times [t_1,(t_1+t_2)/2)),\\
         (\alpha_1)_{[s]}    \quad &((s,t)\in I\times ((t_1+t_2)/2,t_2]).
      \end{cases}
   \end{align*}
   Then, we define $S_0 \cap D_1 \times (I\times \{(t_1+t_2)/2\})$ as the $2m$-braid $\alpha_1^*$ with bands such that the surgery result of $\alpha_1^*$ is $\alpha_1$ (see Figure~\ref{Figure: Thm1-9 motion picture oover E1}).
   We denote by $B_1^-$ the set of these bands.

   \begin{figure}[h]
      \centering
      \includegraphics[width = \hsize]{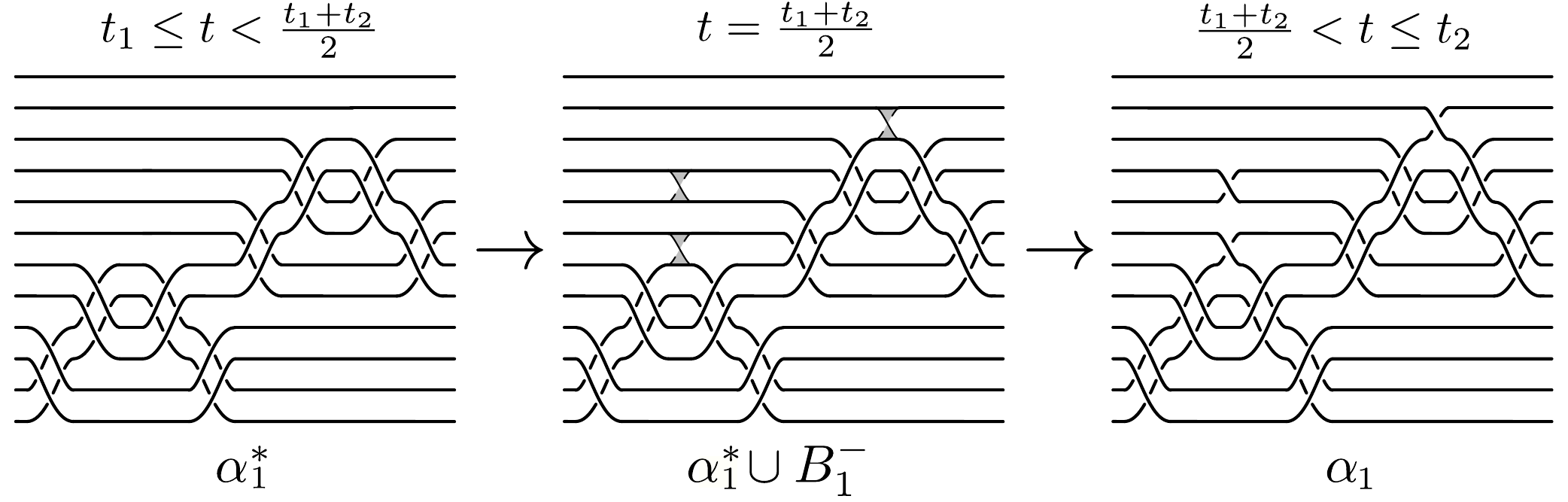}
      \caption{A motion picture of $S_0$ ($t_1\leq t \leq t_2$).}
      \label{Figure: Thm1-9 motion picture oover E1}
   \end{figure}

   \item We construct $S_0\cap D_1\times E_2$ similarly to the case (0).
   First, we define $S_0 \cap D_1 \times \partial E_2$ by
   \begin{align*}
      \mathrm{pr_1}(S_0\cap\mathrm{pr}_2^{-1}(s,t)) ~=~
      \begin{cases}
         (\beta_1)_{[s]}              \quad &((s,t)\in I\times \{t_3\}),\\
         \gamma'_{[(t-t_2)/(t_3-t_2)]} \quad &((s,t)\in \{1\}\times [t_2,t_3]),\\
         (\alpha_1)_{[s]}             \quad &((s,t)\in I\times \{t_2\}),\\
         \gamma_{[(t-t_3)/(t_2-t_3)]}\quad &((s,t)\in \{0\}\times [t_2,t_3]).
      \end{cases}
   \end{align*}
   Since $\beta_1 = \gamma\,\alpha_1\,\gamma'$, the closed braid $S_0 \cap D_1 \times \partial E_2$ is equivalent to the trivial closed braid in $D_1 \times \partial E_2$.
   Thus we may define $S_0 \cap D_1 \times E_2$ as a braided surface of degree $2m$ without branch points.

   \item We construct $S_0 \cap D_1 \times E_3$ similarly to the case (1).
   First, we define $S_0 \cap D_1\times (E_3\setminus I\times \{(t_3+t_4)/2\})$ by
   \begin{align*}
      \mathrm{pr_1}(S_0 \cap\mathrm{pr}_2^{-1}(s,t)) ~=~
      \begin{cases}
         (\beta_1)_{[s]}        \quad &((s,t)\in I\times [t_3,(t_3+t_4)/2)),\\
         (\beta_2)_{[s]}    \quad &((s,t)\in I\times ((t_3+t_4)/2,t_4]).
      \end{cases}
   \end{align*}
   Then, we define $S_0 \cap D_1\times (I\times \{(t_3+t_4)/2\})$ as the $2m$-braid $\beta_1$ with bands belonging to $B_1$.

   \item We construct $S_0\cap D_1 \times E_4$ similarly to the case (2).
   We define $S_0\cap D_1\times \partial E_4$ by
   \begin{align*}
      \mathrm{pr_1}(S_0\cap\mathrm{pr}_2^{-1}(s,t)) ~=~
      \begin{cases}
         (\beta_2)_{[s]}              \quad &((s,t)\in I\times \{t_4\}),\\
         \delta'_{[(t-t_5)/(t_4-t_5)]} \quad &((s,t)\in \{1\}\times [t_4,t_5]),\\
         (\alpha_2)_{[s]}             \quad &((s,t)\in I\times \{t_5\}),\\
         \delta_{[(t-t_4)/(t_4-t_5)]}\quad &((s,t)\in \{0\}\times [t_4,t_5]).
      \end{cases}
   \end{align*}
   Since $\beta_2 = \delta\,\alpha_2\,\delta'$, we define $S_0 \cap D_1\times E_4$ as a braided surface of degree $2m$ without branch points.

   \item We construct $S_0 \cap D_1 \times E_5$ similarly to the case (1).
   We define $S_0 \cap D_1 \times (E_5\setminus \{(t_5 + t_6)/2\})$ by
   \begin{align*}
      \mathrm{pr_1}(S_0 \cap\mathrm{pr}_2^{-1}(s,t)) ~=~
      \begin{cases}
         (\alpha_2)_{[s]}   \quad &((s,t)\in I\times [t_5,(t_5+t_6)/2)),\\
         (\alpha_2^*)_{[s]} \quad &((s,t)\in I\times ((t_5+t_6)/2,t_6]).
      \end{cases}
   \end{align*}
   Then, we define $S_0 \cap D_1\times I\times \{(t_5+t_6)/2\}$ as the $2m$-braid $\alpha_2^*$ with bands attaching to $\alpha_2^*$ as in the opposite direction of Figure~\ref{Figure: Thm1-9 motion picture oover E1} such that the surgery result of $\alpha_2^*$ is $\alpha_2$.
   We denote by $B_1^+$ the set of these bands.

   \item We construct $S_0\cap D_1 \times E_6$ similarly to the case (0).
   First, we define $S_0 \cap D_1 \times \partial E_6$ by
   \begin{align*}
      \mathrm{pr_1}(S_0\cap\mathrm{pr}_2^{-1}(s,t)) ~=~
      \begin{cases}
         (\alpha_2^*)_{[s]}  \quad &((s,t)\in I\times \{t_7\}),\\
         Q_{2m}            \quad &((s,t)\in \{0,1\}\times [t_6,t_7]),\\
         Q_{2m}            \quad &((s,t)\in I\times \{t_6\}).
      \end{cases}
   \end{align*}
   Since $\alpha_2^*$ is equivalent to the trivial $2m$-braid, we may define $S_0 \cap D_1 \times E_6$ as a braided surface of degree $2m$ without branch points.
\end{enumerate}

As a result, we have a properly embedded surface $S_0$ in $D_1\times D_2$.
We take a based point $y_0 = (0,0) \in \partial D_2$.
Then, $S_0$ is a braided surface of degree $2m$ except in neighborhoods of the bands appearing in (1), (3), and (5).
By an ambient isotopy of a neighborhood of each band, we can change the band to a branch point as shown in Figure~\ref{Figure: MPs of band and branch}.
Hence, we obtain a braided surface $S$ of degree $2m$ from $S_0$.
The braided surface $S$ is adequate because the $2m$-braid $\beta_S$ is the composition of adequate $2m$-braids
$\gamma^{-1}$, $\delta$, $\delta'$, and $\gamma'^{-1}$.

\begin{figure}[h]
   \centering
   \includegraphics[width=0.8\hsize]{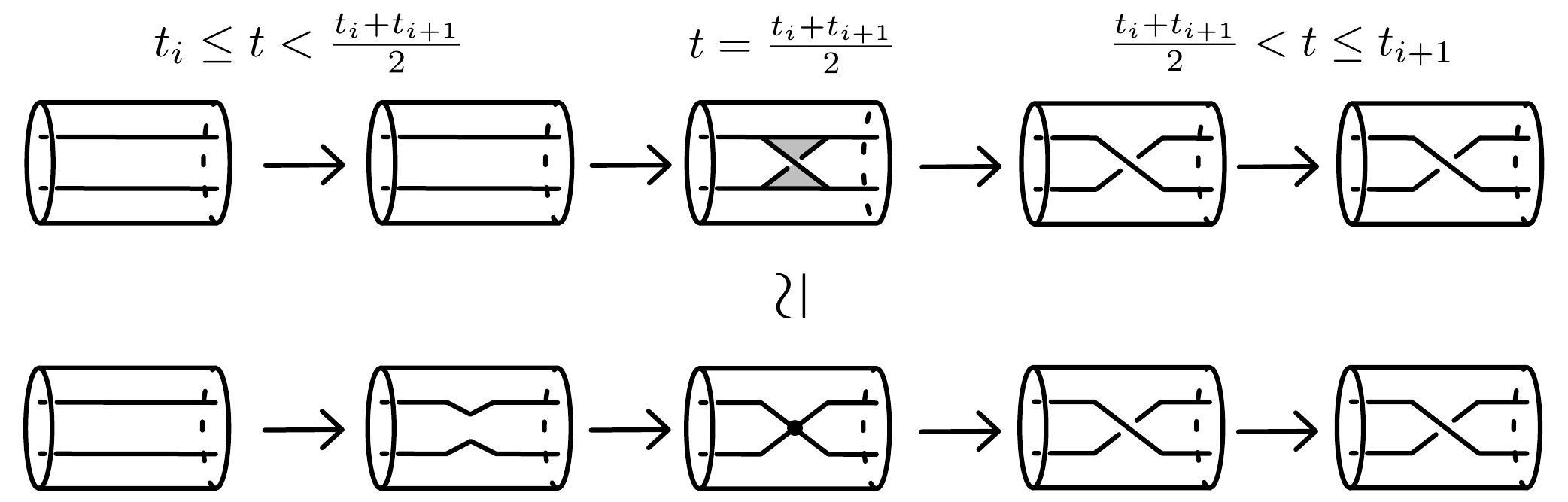}
   \caption{An isotopic deformation changing a band to a branch point.}
   \label{Figure: MPs of band and branch}
\end{figure}

\textbf{Step $3$}: Finally, we show that the surface-link $F$ is equivalent to the plat closure $\widetilde{S}$ of $S$.

Let $p:\R^4=\R^3\times\R \to \R^3$ and $h: \R^4=\R^3\times \R \to \R$ be the projections onto the first and second factors, respectively.
We regard $h$ as a height function of $\R^4$.
Let $A$ be the surface of wicket type associated with $S$.
Note that $\partial A = \partial S = \partial S_0$.
Let $F_0 = S_0\cup A$.
Then $F_0$ is a surface-link equivalent to $\widetilde{S} = S \cup A$.
Thus we show that $F_0$ and $F$ are equivalent.

By an ambient isotopy of $\R^4$ keeping $\R^3 \times (t_0, t_7)$ fixed pointwise, we deform $F_0$ to a surface-link $F_1$ such that
\begin{align*}
   p(F_1\cap \R^3\times \{t\}) ~=~
   \begin{cases}
      \mathbf{D}_1  & (t=t_7),\\
      \widetilde{\alpha_2^*} & ((t_5+t_6)/2 < t < t_7),\\
      \widetilde{\alpha_2^*} \cup |B_1^+| & (t = (t_5+t_6)/2),\\
      \mathrm{p}(F_0 \cap \R^3\times \{t\})  & ((t_1+t_2)/2 < t < (t_5+t_6)/2),\\
      \widetilde{\alpha_1^*} \cup |B_1^-| & (t = (t_1+t_2)/2),\\
      \widetilde{\alpha_1^*} & (t_0 < t < (t_1+t_2)/2),\\
      \mathbf{d}_1   & (t=t_0),\\
      \emptyset  & \textrm{otherwise},
   \end{cases}
\end{align*}
where $|B_1^-|$ (resp. $|B_1^+|$) is the union of the bands belonging to $B_1^-$ (resp. $B_1^+$), and $\mathbf{d}_1$ (resp. $\mathbf{D}_1$) is a union of mutually disjoint $m$ $2$-disks in $\R^3$ bounded by $\widetilde{\alpha_1^*}$ (resp. $\widetilde{\alpha_2^*}$) such that $\mathbf{d}_1$ (resp. $\mathbf{D}_1$) is disjoint from $|B_1^-|$ (resp. $|B_1^+|$) as in the left of Figure~\ref{Figure: Thm1-4 The minimal disks at t=3} except for the attaching arcs of the bands, respectively.
Next, we define a surface-link $F_2$ in $\R^4$ by
\begin{align*}
   p(F_2\cap \R^3\times \{t\}) ~=~
   \begin{cases}
      \mathbf{D}_2 & (t=t_7),\\
      \widetilde{\alpha_2} & ((t_5+t_6)/2 \leq t < t_7),\\
      \mathrm{p}(F_0 \cap \R^3\times \{t\}) & ((t_1+t_2)/2 < t < (t_5+t_6)/2),\\
      \widetilde{\alpha_1} & (t_0 < t \leq (t_1+t_2)/2),\\
      \mathbf{d}_2 & (t=t_0),\\
      \emptyset & \textrm{otherwise},
   \end{cases}
\end{align*}
where $\mathbf{d}_2 = \mathbf{d}_1 \cup |B_1^-|$ (resp. $\mathbf{D}_2 = \mathbf{D}_1 \cup |B_1^+|$) is the union of mutually disjoint $c_1$ (resp. $c_2$) $2$-disks as in the right of Figure~\ref{Figure: Thm1-4 The minimal disks at t=3}, respectively.
\begin{figure}[h]
   \centering
   \includegraphics[width = \hsize]{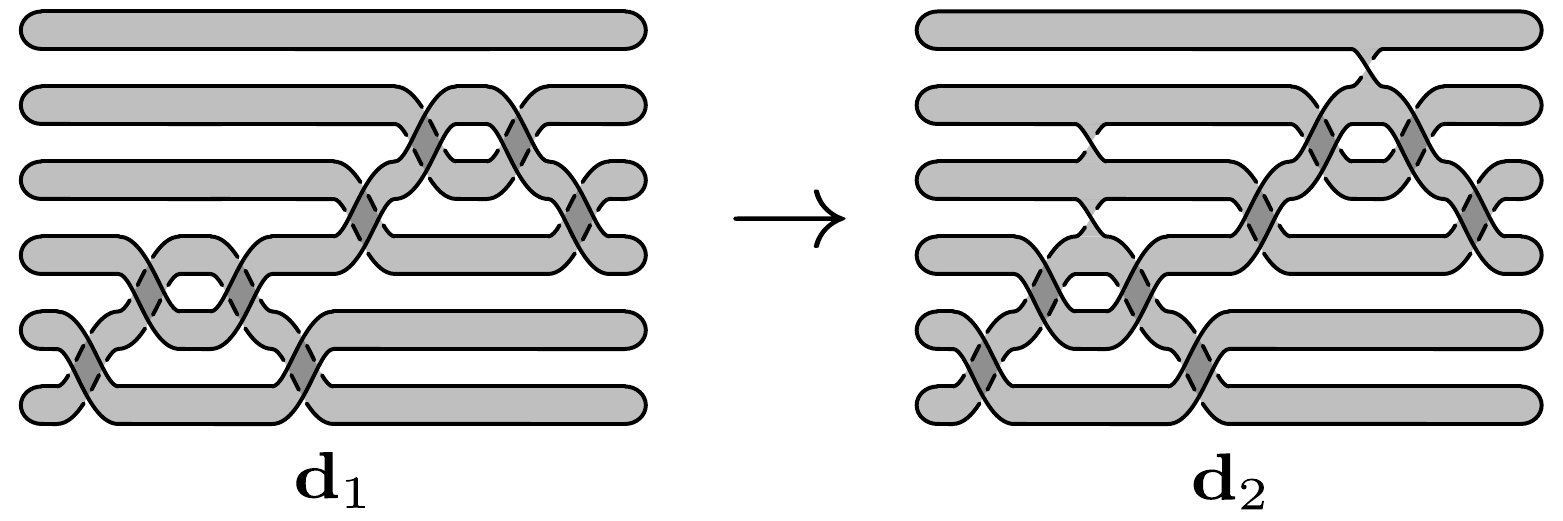}
   \caption{$\mathbf{d}_2$ is the union of $\mathbf{d}_1$ and $|B_1^-|$.}
   \label{Figure: Thm1-4 The minimal disks at t=3}
\end{figure}

Then, $F_2$ is obtained from $F_1$ by cellular moves (cf. \cite{Rourke-Sanderson1982}) along $3$-cells $|B_1^-|\times [t_0,(t_1+t_2)/2] \cup |B_1^+|\times [(t_5+t_6)/2,t_7]$.
This implies that $F_1$ and $F_2$ are equivalent.

Note that $F_2 \cap \R^2 \times \{t_0\} = \mathbf{d}_2 \times \{t_0\}$ is the union of all minimal disks of $F$ with respect to the height function $h$, $F_2 \cap \R^2 \times \{t_7\} = \mathbf{D}_2 \times \{t_7\}$ is the union of all maximal disks of $F$, and all saddle bands of $F$ appear at $t = (t_3 + t_4)/2$ as bands belonging to $B_1$.
By Lemma~\ref{Lemma: 2017_book-thm3.3.2}, $F_2$ is equivalent to a closed realizing surface of the banded link $(\widetilde{\beta_1}, B_1)$.

Since $(L_0,B_0)$ is ambient isotopic to $(\widetilde{\beta_1},B_1)$ and $F$ is equivalent to a closed realizing surface of $(L_0,B_0)$, we see that $F_2$ is equivalent to $F$.
\end{proof}

Next, we show Theorem\ref{Main Theorem-B The existence of a genuine plat form}.
We define the $2m$-braid $\Delta_m$ by $\Delta_1 = Q_2\times I$ and $\Delta_m = \prod_{k=1}^{m-1} (\sigma_{2k}\,\sigma_{2k-1}\cdots\sigma_2\,\sigma_1)$ for $m\geq 2$.
See Figure~\ref{Figure: Delta_1-2-3}.

\begin{figure}[h]
   \centering
   \includegraphics[width = \hsize]{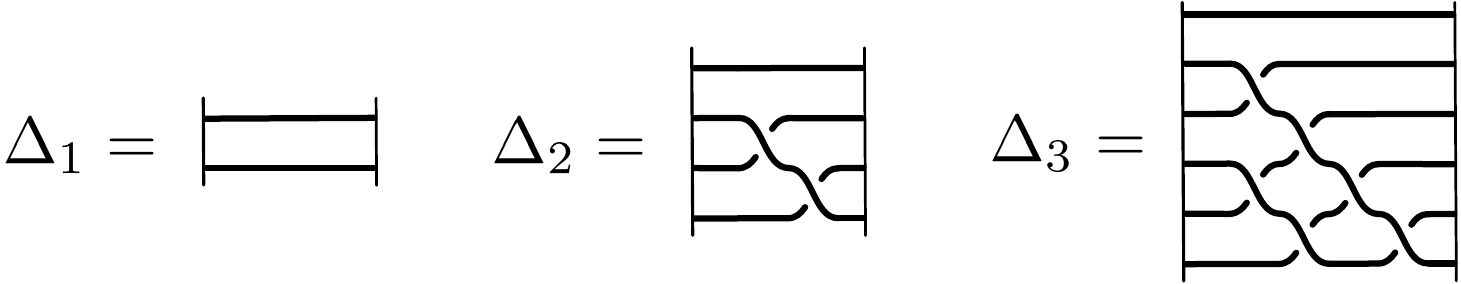}
   \caption{The $2m$-braids $\Delta_{m}$ ($m = 1,2,3$).}
   \label{Figure: Delta_1-2-3}
\end{figure}

Note that the closure of an $m$-braid $b$ is equivalent to the plat closure of a $2m$-braid $\Delta_m\, \iota_{m}^{2m}(b)\, \Delta_m^{-1}$.
See Figures~\ref{Figure: Thm2-5 deformation of Delta_m} and \ref{Figure: Thm2-6 deformation of each side of b}.

\begin{figure}[h]
   \centering
   \includegraphics[width = \hsize]{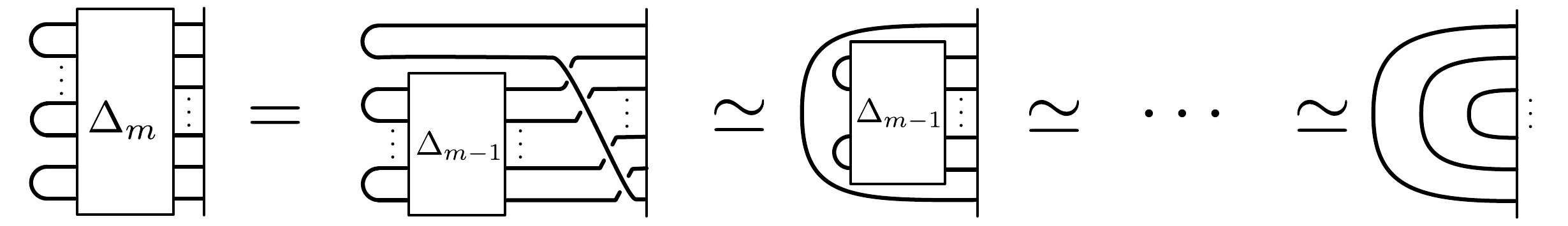}
   \caption{An isotopic deformation of $\Delta_{m}$ with the standard wicket configuration $w_0$ to a configuration of wickets appearing in a closed braid form (Figure~\ref{Figure: The closure of a braid}).}
   \label{Figure: Thm2-5 deformation of Delta_m}
\end{figure}

\begin{figure}[h]
   \centering
   \includegraphics[height=15mm]{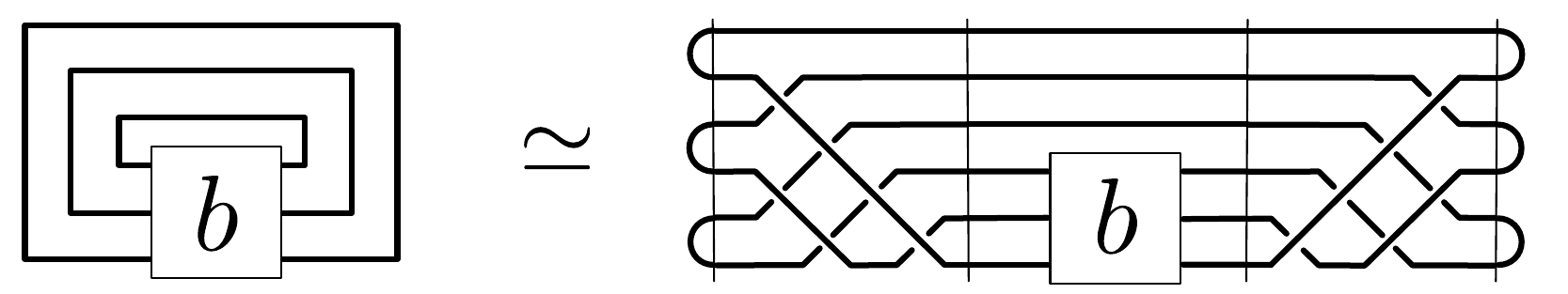}
   \caption{A transformation from the closure of $b$ to the plat closure of $\Delta_m\, \iota_{m}^{2m}(b)\, \Delta_m^{-1}$ ($m=3$).}
   \label{Figure: Thm2-6 deformation of each side of b}
\end{figure}

\begin{proof}[Proof of Theorem~\ref{Main Theorem-B The existence of a genuine plat form}]
Let $F$ be an orientable surface-link.
By Proposition~\ref{Proposition: Alexander theorem wrt closed 2-dim braids}, there exists a $2$-dimensional braid $S$ in $D_1 \times D_2 = D_1 \times I \times I$ whose closure in $\R^4$ is equivalent to $F$.
Let $m$ be the degree of $S$ and $S_{[t]}$ the cross-section $S\cap D_1\times (I\times \{t\})$ for each $t\in I$.
See Figure~\ref{Figure: mp of the closure of 2-dim braid} when $m =3$.
Let $S_1$ be the $2$-dimensional braid of degree $2m$ obtained from $S$ by adding trivial $m$ sheets.

Let $\varepsilon$ be a positive number and let $D'_2 = I \times [-\varepsilon, 1+\varepsilon]$.
We consider a 2-dimensional braid $S_2$ of degree $2m$ in $D_1\times D_2' = D_1\times (I\times [-\varepsilon,1+\varepsilon])$ with a motion picture $(S_2)_{[t]}$ as in Figure~\ref{Figure: Thm2-8 motion picture of S2}.
Here, the motion picture $(S_2)_{[t]}$ for $t\in [-\varepsilon,0]$ (or $t\in [1,1+\varepsilon]$) is the $1$-parameter family of $2m$-braids changing $1_{2m}$ to $\Delta_m\,\Delta_m^{-1}$ (or $\Delta_m\,\Delta_m^{-1}$ to $1_{2m}$), respectively, and the motion picture $(S_2)_{[t]}$ for $t\in I = [0,1]$ is the composition of $\Delta_m$, $(S_1)_{[t]}$ and $\Delta_m^{-1}$.


\begin{figure}[h]
   \centering
   \includegraphics[width = \hsize]{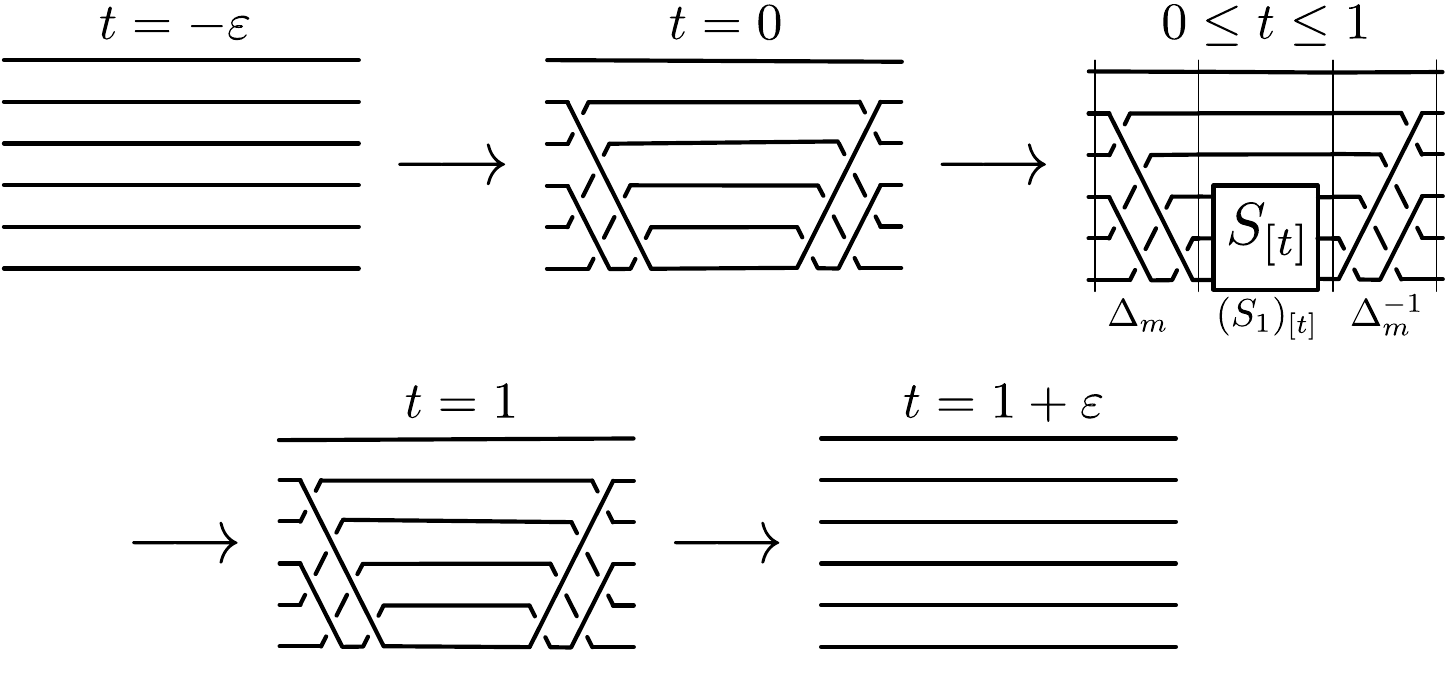}
   \caption{A motion picture of $S_2$ ($m=3$).}
   \label{Figure: Thm2-8 motion picture of S2}
\end{figure}

As a result, the plat closure of $S_2$ has the motion picture as in Figure~\ref{Figure: Thm2-9 motion picture of the plat closure of S2}.
By comparing Figure~\ref{Figure: mp of the closure of 2-dim braid} and Figure~\ref{Figure: Thm2-9 motion picture of the plat closure of S2}, we see that the closure of $S$ is equivalent to the plat closure of $S_2$.
Hence, $F$ has a genuine plat form presentation.
\end{proof}

\begin{figure}[h]
   \centering
   \includegraphics[width = \hsize]{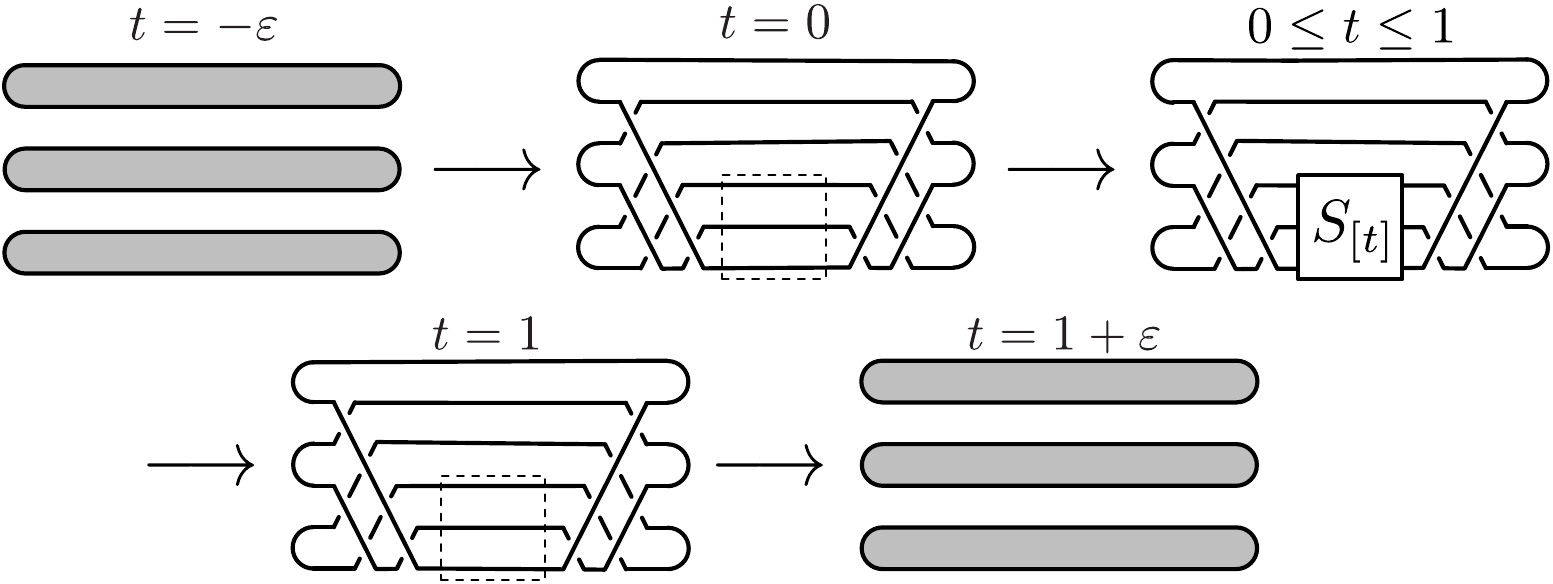}
   \caption{A motion picture of $\widetilde{S_2}$ ($m = 3$).}
   \label{Figure: Thm2-9 motion picture of the plat closure of S2}
\end{figure}

\begin{remark}
   In Lemma~\ref{Lemma: Deformation to normal banded braid form}, each subcylinder $U_i$ contains a part of a banded link as in the left of Figure~\ref{Figure:remark-4}.
   However, we may assume that for each $i$, the band in $U_i$ is either as in the left or as in the right of Figure~\ref{Figure:remark-4}.
   Then we have another braided surface in the proof of Theorem~\ref{Main theorem-A The existence of a plat form}, where the corresponding branch point changes the sign.
   (A branch point of a braided surface is \textit{positive} (or \textit{negative}) if the local monodromy is a conjugate of a standard generator (or its inverse), cf. \cite{Kamada2002_book, Kamada2017_book}).

   \begin{figure}[h]
      \centering
      \includegraphics[height = 25mm]{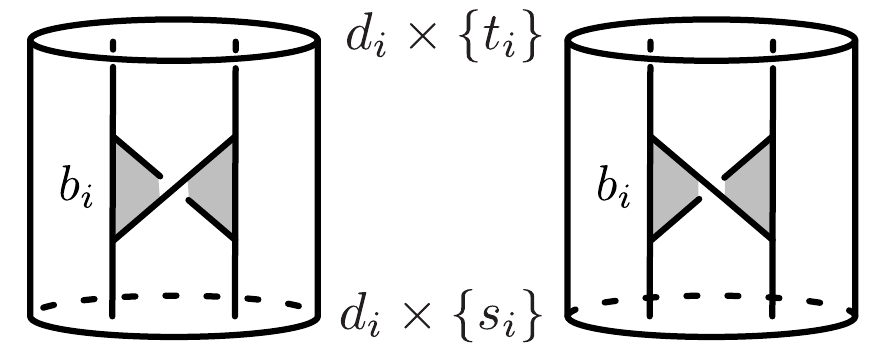}
      \caption{Two types of half-twisted bands in a subcylinder $U_i = d_i \times [s_i, t_i]$.}
      \label{Figure:remark-4}
   \end{figure}
\end{remark}

\section{The plat index of a surface-link and examples}\label{Section: The plat index of surface-links and examples}
In this section, we introduce two surface-link invariants called the plat index and the genuine plat index.

\begin{definition}
   Let $F$ be a surface-link.
   The \textit{plat index of $F$}, denoted by $\mathrm{Plat}(F)$, is defined as the half of the minimum degree of all adequate braided surfaces whose plat closures are equivalent to $F$:
   \[
      \mathrm{Plat}(F) ~=~ \min \{\,\deg S / 2 \,|\, S~\mbox{is a braided surface with } \widetilde{S} \simeq F \}.
   \]
\end{definition}

\begin{definition}
   Let $F$ be a surface-link.
   If $F$ admits a genuine plat form, the \textit{genuine plat index of $F$}, denoted by $\mathrm{g.Plat}(F)$, is defined as the half of the minimum degree of all $2$-dimensional braids whose plat closures are equivalent to $F$.
   If $F$ dose not admit a genuine plat form, it is defined as infinity:
   \[
      \mathrm{g.Plat}(F) ~=~
      \begin{cases}
         \min \{\deg S / 2 \,|\, S~\mbox{is a}~2\mathchar`-\mbox{dimensional braid with } \widetilde{S} \simeq F \},\\
         \infty \quad \mbox{if } F \mbox{ admits no genuine plat forms}.\\
      \end{cases}
   \]
\end{definition}
By definition, it holds that $\mathrm{Plat}(F) \leq \mathrm{g.Plat}(F)$ for every surface-link $F$.
Moreover, from the proof of Theorem~\ref{Main Theorem-B The existence of a genuine plat form}, we have the following proposition.

\begin{proposition}\label{Proposition inequality of plat indices}
   The following inequalities hold for every orientable surface-link $F$:
   \[
      \mathrm{Plat}(F) ~\leq~ \mathrm{g.Plat}(F) ~\leq~ \mathrm{Braid}(F).
   \]
\end{proposition}

In the rest of this section, we show some examples of surface-links in plat forms and discuss the plat index and the genuine plat index.

We first recall the notion of a braid system of a braided surface.
Refer to \cite{Kamada2002_book} for more details.
Let $\mathrm{pr}_i: D_1\times D_2 \to D_i$ ($i=1,2$) be the projection and $\mathcal{C}_n$ the configuration space of $n$ points of $\Int D_1$.
Let $S$ be a braided surface of degree $n$, and $\Sigma(S)$ the branch locus of $\pi_S = \mathrm{pr}_2|_S: S \to D_2$.
Let $y_0 \in \partial D_2$ be a fixed base point.

The \textit{braid monodromy} of $S$ is a homomorphism $\rho_S: \pi_1(D_2\setminus \Sigma(S), y_0) \to \pi_1(\mathcal{C}_n, Q_n) = B_n$ defined as follows:
For a loop $c: (I,\partial I) \to (D_2\setminus \Sigma(S), y_0)$, define a loop $\rho_S(c): (I,\partial I) \to (\mathcal{C}_n, Q_n)$ as $\rho_S(c)(t) = \mathrm{pr}_1(\pi_S^{-1}(c(t)))$.
Then the braid monodromy of $S$ is defined as a group homomorphism sending $[c]$ to $[\rho_S(c)] \in \pi_1(\mathcal{C}_n, Q_n)$.

Let $r$ be a positive integer.
A \textit{Hurwitz arc system} in $D_2$ (with the base point $y_0$) is an $r$-tuple $\mathcal{A} = (\alpha_1, \cdots, \alpha_r)$ of oriented simple arcs in $D_2$ such that
\begin{enumerate}
   \item for each $i$, $\alpha_i \cap \partial D_2 = \partial \alpha_i \cap \partial D_2 = \{y_0\}$ and this is the terminal point of $\alpha_i$,
   \item for $i \neq j$, $\alpha_i \cap \alpha_j = \{y_0\}$, and
   \item $\alpha_1, \ldots,  \alpha_r$ appear in this order around the base point $y_0$.
\end{enumerate}
The set of initial points of $\alpha_1, \ldots, \alpha_r$ is called the \textit{starting point set} of $\mathcal{A}$.

Let $\mathcal{A} = (\alpha_1, \cdots, \alpha_r)$ be a Hurwitz arc system with the starting point set $\Sigma(S)$.
For each $i$, let $N_i$ be a (small) regular neighborhood of the starting point of $\alpha_i$, $\overline{\alpha_i}$ an oriented arc obtained from $\alpha_i$ by restricting to $D_2\setminus \Int N_i$, and $\gamma_i$ a loop $\overline{\alpha_i}^{-1} \cdot \partial N_i \cdot \overline{\alpha_i}$ in $D_2\setminus \Sigma(S)$ with base point $y_0$.
Here, $\partial N_i$ is oriented counter-clockwise.
Then $\pi_1(D_2\setminus \Sigma(S), y_0)$ is generated by $[\gamma_1], [\gamma_2], \ldots, [\gamma_r]$ and we have $[\partial D_2] = [\gamma_1] \cdots [\gamma_r]$.
The \textit{braid system} of $S$ associated with $\mathcal{A}$ is an $r$-tuple $b_S$ of elements of $B_n$ defined as
\[
   b_S ~=~ (\rho_S([\gamma_1]), \ldots, \rho_S([\gamma_r])) \in (B_n)^r.
\]

It is known that $\rho_S([\gamma_i])$ is a conjugation of a standard generator or its inverse, $\sigma_1^{\,\varepsilon}$ ($\varepsilon \in \{\pm 1\}$), such that $\varepsilon$ is the sign of the branch point which is the starting point of $\alpha_i$.
The composition $\rho_S([\gamma_1])\rho_S([\gamma_2]) \cdots \rho_S([\gamma_r])$ is equal to $\beta_S$ in $B_n$.

The \textit{slide action} of the braid group $B_r$ on $(B_n)^r$ is a left group action defined as
\[
   \mathrm{slide}(\sigma_j)(\beta_1, \ldots, \beta_r) ~=~ (\beta_1, \ldots, \beta_{j-1}, \beta_j\beta_{j+1}\beta_j^{-1}, \beta_j, \beta_{j+2}, \ldots, \beta_r)
\]
for $\sigma_j\in B_r$ and $(\beta_1, \ldots, \beta_r)\in (B_n)^r$.
Two elements of $(B_n)^r$ are said to be \textit{Hurwitz equivalent} if they are in the same orbit of the slide action of $B_r$.

\begin{lemma}[cf. \cite{Kamada2002_book,Moishezon1981,Rudolph1985}]\label{Lemma: Hurwitz equivalence}
   Two braided surfaces in $D_1\times D_2$ are equivalent if and only if their braid systems are Hurwitz equivalent.
\end{lemma}

Let $e(F)$ be the normal Euler number of a surface-knot $F$. 
The normal Euler number of any orientable surface-knot is $0$, and the normal Euler number of a trivial non-orientable surface-knot, which is a connected sum of $p$ copies of $P_+$ and $q$ copies of $P_-$, is $2(p-q)$ (cf. \cite{Carter-Saito1992,Hosokawa-Kawauchi1979}).

\begin{theorem}\label{Theorem The classification of trivial surface-knots}
   Let $F$ be a surface-link.
   \begin{enumerate}
      \item $\mathrm{Plat}(F) = 1$ if and only if $F$ is either a trivial $2$-knot or a trivial non-orientable surface-knot.
      \item $\mathrm{g.Plat}(F) = 1$ if and only if $F$ is either a trivial $2$-knot or a trivial non-orientable surface-knot with $e(F) = 0$.
      \item If $F$ is a trivial orientable surface-knot with positive genus, then $\mathrm{Plat}(F) = \mathrm{g.Plat}(F) = 2$.
   \end{enumerate}
\end{theorem}

\begin{proof}
(1) Let $F$ be a surface-link with $\mathrm{Plat}(F) = 1$ and $S$ a braided surface of degree $2$ with $\widetilde{S} \simeq F$.
Let $p$ and $q$ be the numbers of positive and negative branch points of $S$, respectively.
Then a braid system for $S$ is Hurwitz equivalent to $(\sigma_1, \ldots, \sigma_1, \sigma_1^{-1}, \ldots, \sigma_1^{-1})$ consisting of $p$ $\sigma_1$'s and $q$ $\sigma_1^{-1}$'s.
Hence, the equivalence class of $S$ is determined from $p$ and $q$.
Figure~\ref{Figure: mp of 1-plat} is a motion picture of the plat closure of a braided surface of degree $2$ with $p$ positive branch points and $q$ negative branch points.
The motion picture describes a trivial $2$-knot if $p = q = 0$ holds, otherwise, it describes a connected sum of $p$ copies of $P_+$ and $q$ copies of $P_-$.
Therefore, $F$ is either a trivial 2-knot or a trivial non-orientable surface-knot.

Conversely, if $F$ is a trivial $2$-knot, then $\mathrm{Plat}(F) = 1$.
If $F$ is a trivial non-orientable surface-knot, then $F$ is equivalent to a surface-knot described in Figure~\ref{Figure: mp of 1-plat} and hence $\mathrm{Plat}(F) = 1$.
\begin{figure}[h]
   \centering
   \includegraphics[height = 25mm]{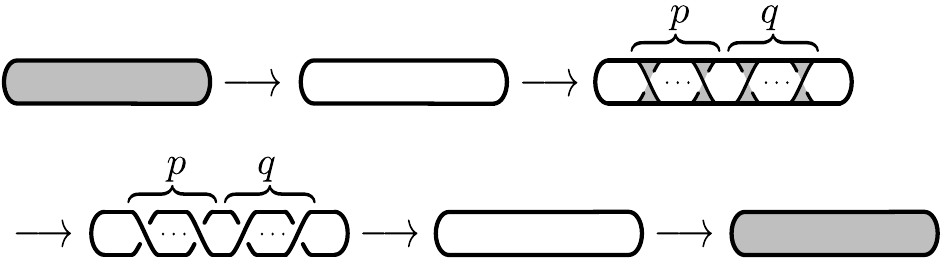}
   \caption{A surface-knot in a (normal) plat form.}
   \label{Figure: mp of 1-plat}
\end{figure}

(2) Let $F$ be a surface-link with $\mathrm{g.Plat}(F) = 1$ and $S$ a $2$-dimensional braid of degree 2 with $\widetilde{S} \simeq F$.
Since $S$ is a $2$-dimensional braid, the number of positive branch points of $S$, denoted by $p$, is equal to the number of negative ones.
Hence, the argument in the proof of (1) implies that $F$ is equivalent to a trivial $2$-knot, when $p = 0$, or a connected sum of $p$ copies of $P_+$ and $p$ copies of $P_-$.
In particular, it holds that $e(F) = 0$.

Conversely, if $F$ is a trivial $2$-knot, then $\mathrm{g.Plat}(F) = 1$.
If $F$ is a trivial non-orientable surface-knot with $e(F) = 0$, then $F$ is a connected sum of $p$ copies of $P_+$ and $p$ copies of $P_-$ for some $p > 0$, which is equivalent to a surface-knot described in Figure~\ref{Figure: mp of 1-plat} with $p = q$.
Hence $\mathrm{g.Plat}(F) = 1$.

(3) Let $F$ be a trivial orientable surface-knot with a positive genus.
Since $\mathrm{Braid}(F) = 2$ (cf. \cite{Kamada1992,Kamada2002_book}), by Proposition~\ref{Proposition inequality of plat indices}, we have $\mathrm{Plat}(F) \leq \mathrm{g.Plat}(F) \leq 2$.
On the other hand, by (1), it holds that $\mathrm{Plat}(F) \neq 1$.
Hence, we have $\mathrm{Plat}(F) = \mathrm{g.Plat}(F) = 2$.
(Figure~\ref{Figure: Descriptions of genuine 1-plat, Sigma-g} shows a motion picture of a genuine plat form of $F$.)
\end{proof}

\begin{figure}[h]
   \centering
   \includegraphics[width = \hsize]{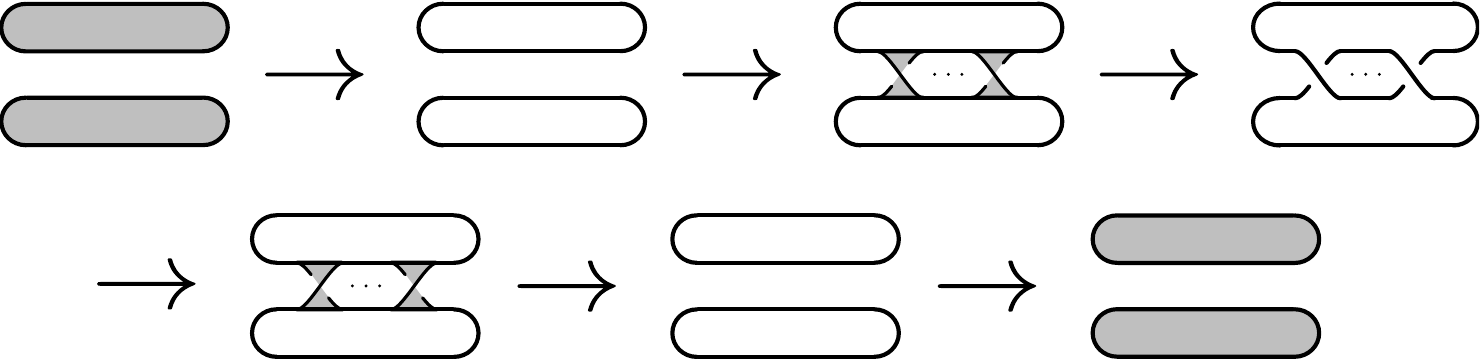}
   \caption{A trivial orientable surface-knot with a positive genus in a genuine plat form.}
   \label{Figure: Descriptions of genuine 1-plat, Sigma-g}
\end{figure}

\begin{proposition}
   Let $F$ be the $2$-knot denoted by $2\_2$ in the table of \cite{Kanenobu-Takahashi2020}, which is depicted in Figure~\ref{Figure: mp of 2-2}.
   Then $\mathrm{Plat}(F) = \mathrm{g.Plat}(F) = 2$.
\end{proposition}

\begin{proof}
   Figure~\ref{Figure: Description of 2-2} shows a deformation of a banded link by an isotopy of $\R^3$.
   Using the isotopy, we see that $F$ is equivalent to a surface-knot in a genuine plat form depicted in Figure~\ref{Figure: mp of 2-2 in a plat form}.
   Hence, we have the inequality $\mathrm{Plat}(F)\leq \mathrm{g.Plat}(F)\leq 2$.
   Since $F$ is not a trivial $2$-knot, we have $\mathrm{Plat}(F) = \mathrm{g.Plat}(F) = 2$.
\end{proof}

The braid index of every non-trivial surface-knot is greater than $2$ (\cite{Kamada1992}).
Hence, $2\_2$ is an example such that the equality in $\mathrm{g.Plat}(F) \leq \mathrm{Braid}(F)$ in Proposition~\ref{Proposition inequality of plat indices} does not hold.

\begin{figure}[h]
   \centering
   \includegraphics[width = \hsize]{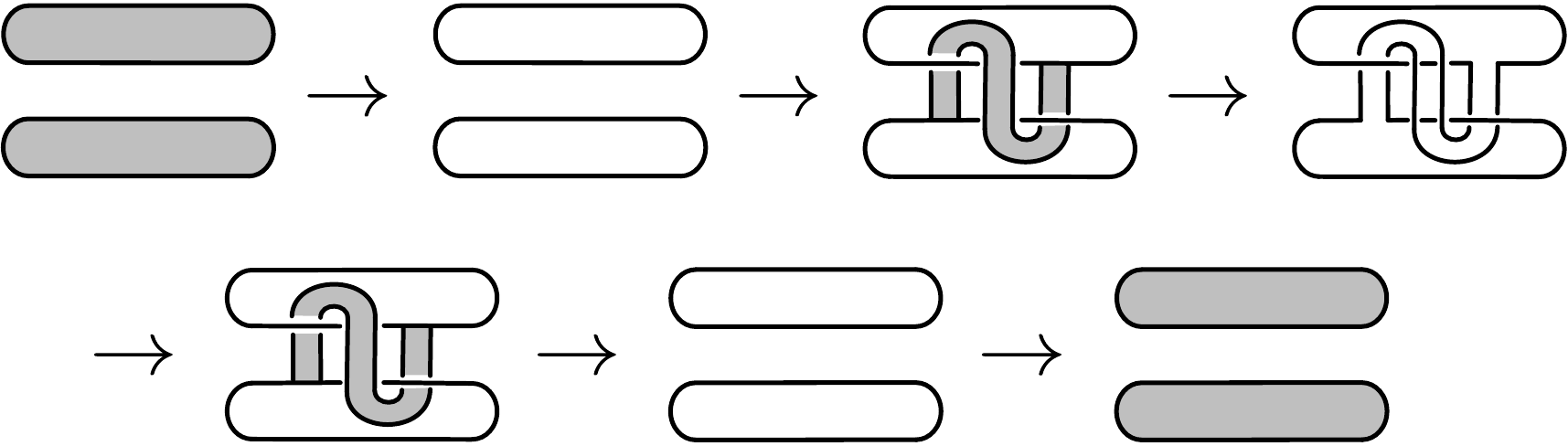}
   \caption{A motion picture of the 2-knot $2\_2$.}
   \label{Figure: mp of 2-2}
\end{figure}

\begin{figure}[h]
   \centering
   \includegraphics[height = 14.5mm]{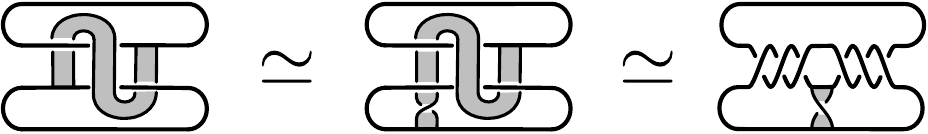}
   \caption{An isotopic deformation of a banded link.}
   \label{Figure: Description of 2-2}
\end{figure}

\begin{figure}[h]
   \centering
   \includegraphics[width = \hsize]{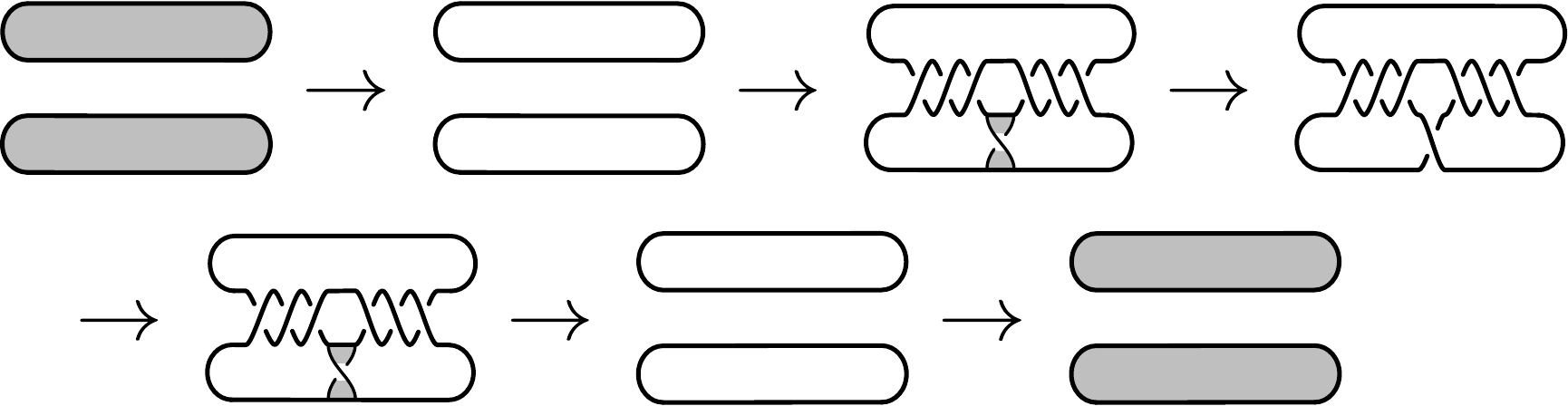}
   \caption{A motion picture of the 2-knot $2\_2$ in a genuine plat form.}
   \label{Figure: mp of 2-2 in a plat form}
\end{figure}

A surface-link is said to be \textit{ribbon} if it is obtained from a trivial $2$-link by some $1$-handle surgeries.

\begin{theorem}\label{Theorem genuine 2-plat 2-knot is ribbon}
   Let $F$ be a $2$-knot (or a surface-link with $\chi(F) = 2$) with $\mathrm{g.Plat}(F) = 2$.
   Then, $F$ is ribbon.
\end{theorem}

\begin{proof}
   Let $S$ be a $2$-dimensional braid of degree $4$ with $\widetilde{S} \simeq F$, and $r$ the number of branch points of $S$.
   Since $\chi(F) = 2$, we see that $r = 2$ from $\chi(\widetilde{S}) = 4-r$.
   Let $b_S = (\beta_1, \beta_2) \in (B_4)^2$ be a braid system of $S$.
   Since $S$ is a $2$-dimensional braid, $\beta_S = \beta_1\beta_2 = 1$ in $B_4$, i.e., $\beta_2 = \beta_1^{-1}$.
   A $2$-dimensional braid with a symmetric braid system $(\beta_1, \beta_1^{-1})$ is known as a ribbon $2$-dimensional braid (\cite{Kamada2002_book}), which is equivalent to a $2$-dimensional braid $S'$ in $D_1\times D_2 = D_1\times (I\times [0,1])$ such that $S'$ is symmetric with respect to $t = 1/2$.
   Then the plat closure of $S'$ is symmetric with respect to $t = 1/2$ and it is in a normal form in the sense of \cite{K-S-S1982}.
   Hence $\widetilde{S'}$ is ribbon (cf. Theorem 11.4 of \cite{Kamada2002_book}).
   Since $F \simeq \widetilde{S}$ and $\widetilde{S} \simeq \widetilde{S'}$, $F$ is ribbon.
\end{proof}



\begin{proposition}\label{Proposition plat index of 2-twist spun twist knot}
   Let $k(n)$ be the twist knot ($n \in \Z$) and $F(n)$ the $2$-twist spin of $k(n)$ (\cite{Zeeman1965}).
   Then $\mathrm{Plat}(F(n)) = 2$ holds for $n \neq 0$.
\end{proposition}

\begin{proof}
   The 2-knot $F(n)$ has a motion picture described in \cite{Kanenobu1983} as in Figure~\ref{Figure: mp of F(n)}, where $m = 2n+1$ and a box labeled by $m$ contains $m$ positive half-twists or $-m$ negative half-twists for $m < 0$.
   Since the trivial link depicted in (4) of Figure~\ref{Figure: mp of F(n)} is the plat closure of an adequate braid of degree $4$, this motion picture gives us a (normal) plat form presentation for $F(n)$.
   On the other hand, it is known that $F(n)$ is a non-trivial $2$-knot if $n \neq 0$.
   Hence, we have that $\mathrm{Plat}(F(n)) = 2$.
\end{proof}

\begin{figure}[h]
   \centering
   \includegraphics[width = \hsize]{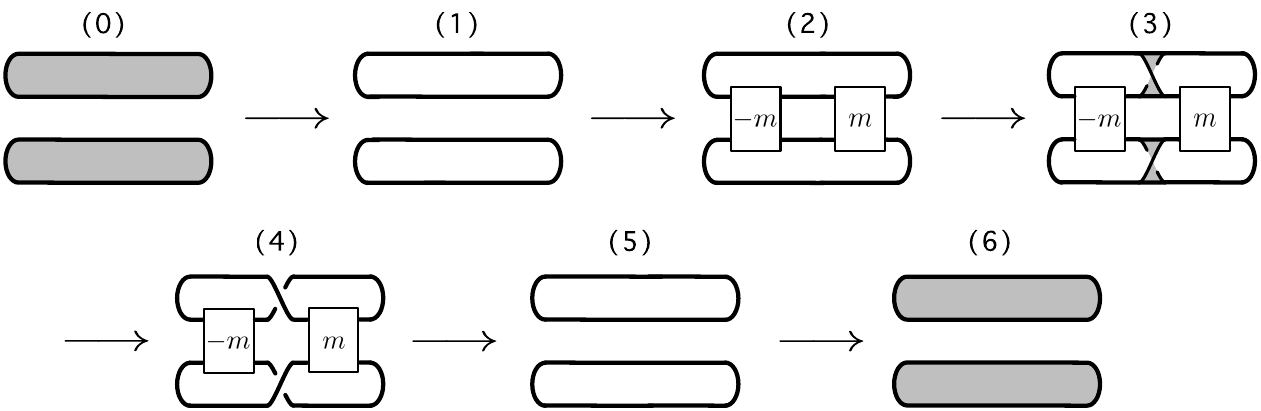}
   \caption{The 2-knot $F(n)$ in a plat form ($m =2n+1$).}
   \label{Figure: mp of F(n)}
\end{figure}

Furthermore, it is known that $F(n)$ is not a ribbon $2$-knot for $n \neq 0$ (\cite{Cochran1983}).
By Theorem~\ref{Theorem genuine 2-plat 2-knot is ribbon}, the genuine plat index of $F(n)$ is greater than $2$.
Thus, Proposition~\ref{Proposition plat index of 2-twist spun twist knot} gives us examples of 2-knots such that the equality in $\mathrm{Plat}(F) \leq \mathrm{g.Plat}(F)$ in Proposition~\ref{Proposition inequality of plat indices} does not hold.

A \textit{$P^2$-link} is a surface-link whose components are projective planes.
Replacing $m = 2n+1$ (or $-m = -2n-1$) crossings in Figure~\ref{Figure: mp of F(n)} with $2n$ (or $-2n$) crossings, respectively, we have a $2$-component $P^2$-link in a plat form.
In particular, in the case of $n = 1$, the $P^2$-link is a $P^2$-link denoted by $8_1^{-1, -1}$ in Yoshikawa's table (\cite{Yoshikawa1994}).

\begin{proposition}\label{Proposition: normal Euler number of genuine plat form}
   Let $F$ be a surface-link in a genuine plat form.
   Each component of $F$ is a surface-knot whose normal Euler number is zero.
\end{proposition}

\begin{proof}
   Each connected component of $F$ is regarded as a surface-knot in a genuine plat form by forgetting other components of $F$.
   Thus it is sufficient to show that $e(F) = 0$ for a surface-knot $F$ in a genuine plat form.

   For a (broken surface) diagram $D$ of $F$ (cf. \cite{Carter-Saito1998}), let $b_+(D)$ (resp. $b_-(D)$) be the number of positive (resp. negative) branch points of $D$. 
   Then, the normal Euler number $e(F)$ is equal to $b_+(D) - b_-(D)$.

   When $F = \widetilde{S}$ is in a genuine plat form, taking a diagram suitably, positive (resp. negative) branch points of $S$ (in the sense of a $2$-dimensional braid) correspond to positive (resp. negative) branch points of $D$, and vise versa.
   Since $S$ is a 2-dimensional braid, the number of positive branch points of $S$ and that of negative branch points of $S$ are the same.
   Thus we have $e(F) = b_+(D) - b_-(D) = 0$.
\end{proof}

It is unknown to the author whether every surface-link consisting of surface-knots whose normal Euler numbers are zero is equivalent to a surface-link in a genuine plat form.

\section*{Acknowledgment}
The author would like to thank Seiichi Kamada and Taizo Kanenobu for their helpful advice on this research.

\bibliographystyle{plain}
\bibliography{reference}
\end{document}